\documentclass[11pt]{amsart}
\usepackage{amsmath,amsthm, amscd, amssymb, amsfonts}
\usepackage[all]{xy}
\usepackage{mathrsfs}
\usepackage{enumitem}
\usepackage[T2A,T1]{fontenc}
\newcommand{\ztt}{\mbox{\usefont{T2A}{\rmdefault}{m}{n}\cyrc}}
\newcommand{\sh}{\mbox{\usefont{T2A}{\rmdefault}{m}{n}\cyrsh}}

\newcommand{\ya}{\mbox{\usefont{T2A}{\rmdefault}{m}{n}\cyrya}}

\newcommand{\sch}{\mbox{\usefont{T2A}{\rmdefault}{m}{n}\cyrshch}}

\usepackage{caption}

\usepackage[ansinew]{inputenc}
\usepackage{graphicx,fancyhdr}

\usepackage[dvips, dvipsnames, usenames]{color}

\newcommand{\comment}[1]{}

\numberwithin{equation}{section}
\theoremstyle{plain}
\newtheorem{theorem}{Theorem}[section]
\newtheorem{lemma}[theorem]{Lemma}
\newtheorem{coro}[theorem]{Corollary}

\newtheorem{prop}[theorem]{Proposition}

\theoremstyle{definition}

\newtheorem{example}[theorem]{Example}

\theoremstyle{remark}

\newtheorem{remark}[theorem]{Remark}
\newtheorem{paso}{Step}

\newtheorem{case}{Case}

\def\pf{\begin{proof}}
\def\epf{\end{proof}}

\newcommand{\fudos}{\hspace{-1pt}\frac{3}{2}}

\newcommand{\vi}{\textbf{(i)} }
\newcommand{\vii}{\textbf{(ii)} }

\newcommand{\rests}{ \mathfrak{s}}
\newcommand{\ba}{ \mathbf{a}}

\newcommand{\bm}{ \mathbf{m}}
\newcommand{\bn}{ \mathbf{n}}
\newcommand{\ku}{ \Bbbk}
\newcommand{\kut}{ \ku^{\times}}
\newcommand{\Fp}{\mathbb F_p}
\newcommand{\G}{\mathbb G}

\newcommand{\ghost}{\mathscr{G}}

\newcommand{\sa}{\mathtt{r}}

\newcommand{\qmb}{\mathtt{q}}
\newcommand{\I}{\mathbb I}
\newcommand{\Iw}{\mathbb I^{\dagger}}
\newcommand{\Idd}{\mathbb I^{\ddagger}}

\newcommand{\N}{\mathbb N}
\newcommand{\bp}{\mathbf{p}}
\newcommand{\bq}{\mathbf{q}}

\newcommand{\Z}{\mathbb Z}

\renewcommand{\_}[1]{_{\left( #1 \right)}}

\newcommand{\cA}{\mathcal{A}}
\newcommand{\cB}{\mathcal{B}}
\newcommand{\cBt}{\widetilde{\mathcal{B}}}

\newcommand{\D}{\mathcal{D}}

\newcommand{\cJ}{\mathcal{J}}

\newcommand{\Ss}{{\mathcal S}}

\newcommand{\cV}{\mathcal{V}}

\newcommand{\X}{\mathcal{X}}

\newcommand{\lstr}{\mathfrak L}
\newcommand{\cyc}{\mathfrak C}
\newcommand{\pos}{\mathfrak P}
\newcommand{\eny}{\mathfrak E}

\newcommand\ad{\operatorname{ad}}
\newcommand{\Alg}{\Hom_{\text{alg}}}

\newcommand{\Der}{\operatorname{Der}}
\newcommand{\diag}{\operatorname{diag}}

\newcommand{\id}{\operatorname{id}}
\newcommand{\gr}{\operatorname{gr}}
\newcommand{\GK}{\operatorname{GKdim}}
\newcommand{\Hom}{\operatorname{Hom}}

\def\ydh{{}^{H}_{H}\mathcal{YD}}

\newcommand{\Bdiag}{\mathcal{B}^\mathrm{diag}}
\newcommand{\Vdiag}{\mathcal{V}^\mathrm{diag}}
\newcommand{\NA}{\mathscr{B}}
\newcommand{\toba}{\mathscr{B}}
\newcommand{\ot}{\otimes}

\newcommand{\ydG}{{}^{\ku \Gamma }_{\ku \Gamma }\mathcal{YD}}
\newcommand{\ydg}{{}^{\ku G}_{\ku G}\mathcal{YD}}
\newcommand{\ydz}{{}^{\ku \Z}_{\ku \Z}\mathcal{YD}}

\allowdisplaybreaks

\newcounter{tabla}\stepcounter{tabla}

\begin{document}

\title[Pointed Hopf algebras in positive characteristic]{Examples of finite-dimensional pointed Hopf algebras in positive characteristic}

\author[Andruskiewitsch, Angiono]
{Nicol\'as Andruskiewitsch, Iv\'an Angiono}

\address{FaMAF-CIEM (CONICET), Universidad Nacional de C\'ordoba,
Medina A\-llen\-de s/n, Ciudad Universitaria, 5000 C\' ordoba, Rep\'
ublica Argentina.} \email{(andrus|angiono)@famaf.unc.edu.ar}

\author[Heckenberger]
{Istv\'an Heckenberger}

\address{Philipps-Universität Marburg,
Fachbereich Mathematik und Informatik,
Hans-Meerwein-Straße,
D-35032 Marburg, Germany.} \email{heckenberger@mathematik.uni-marburg.de}

\dedicatory{To Nikolai Reshetikhin on his 60th birthday with admiration.}

\thanks{The work of N. A. and I. A.  was partially supported by CONICET,
Secyt (UNC). 
The work of N. A. and I. A., respectively I. H., was partially done during visits to the University of Marburg, respectively C\'ordoba, 
supported by the Alexander von Humboldt Foundation
through the Research Group Linkage Programme}

\begin{abstract}
We present new examples of finite-dimensional Nichols algebras over fields of positive characteristic. The corresponding braided vector spaces are not of diagonal type, admit a realization as Yetter-Drinfeld modules over finite abelian groups and are analogous to braidings over fields of characteristic zero whose Nichols algebras have finite Gelfand-Kirillov dimension.

We obtain new examples of finite-dimensional pointed Hopf algebras by bosonization with group algebras of suitable finite abelian groups.
\end{abstract}

\maketitle


\section{Introduction}\label{section:introduction}

\subsection{Overview}
This is a contribution to the classification of finite-dimen\-sional pointed Hopf algebras in positive characteristic. 
Beyond the classical theme of cocommutative Hopf algebras--see for instance \cite{CF} and references therein--the problem was considered in several recent works \cite{clw,HW,NW,NWW1,NWW2,W}.
As in various of these papers, the focus of our work is on finite-dimensional Nichols algebras over finite abelian groups. 
Let $\ku$ be an algebraically closed field of characteristic $p \geq 0$. 
When $p=0$, such Nichols algebras are necessarily of diagonal type and their classification was achieved in \cite{H-classif}. 
When $p> 0$,  finite-dimensional Nichols algebras of diagonal type of rank 2 and 3 were classified in \cite{HW,W}.
Notice that there are more examples than in characteristic 0: indeed, 1 in the diagonal is no longer excluded.

\begin{example}\label{exa:qls}
Assume that $p >0$. Given $\theta\in \N$, we set $\I_{\theta} = \{1,2, \dots, \theta\}$. Let $\bq = (q_{ij})_{i, j \in \I_{\theta}} \in \ku^{\theta \times \theta}$, 
be a matrix with $q_{ii} = 1 = q_{ij}q_{ji}$ for all $i\neq j \in \I_{\theta}$.
Let$(V, c)$ be a braided vector space of dimension $\theta$, of diagonal type with matrix $\bq$ with respect to a basis
$(x_i)_{i\in \I_\theta}$, that is $c: V \otimes V \to V \otimes V$ is given by $c(x_i \otimes x_j) = q_{ij} x_j \otimes x_i$.
Then the corresponding Nichols algebra is 
\begin{align*}
\toba(V) & = \rests_{\bq}(V) := T(V) / \langle x_i^p,\ i \in \I_{\theta}, \quad x_ix_j - q_{ij}x_j x_i, \ i < j \in \I_{\theta}  \rangle.
\end{align*}
Clearly, $\dim \rests_{\bq}(V)  = p^{\theta}$.
\end{example}

Furthermore, if $p>0$, then there are finite-dimensional Nichols algebras over abelian groups that are \emph{not} of diagonal type, a remarkable example being the Jordan plane that has dimension $p^2$ \cite{clw} (it gives rise to pointed Hopf algebras of order $p^3$, see \cite{NW}), in contrast with characteristic 0,
where it has Gelfand-Kirillov dimension $2$. 
In fact, Nichols algebras over abelian groups \emph{with finite Gelfand-Kirillov dimension and assuming $p=0$} 
were the subject of the recent papers \cite{aah-triang,aah-diag}. Succinctly, 
the main relevant results in loc. cit. are:

\begin{itemize}[leftmargin=*] \renewcommand{\labelitemi}{$\circ$}
\item It was conjectured in \cite{aah-triang} that finite GK-dimensional Nichols algebras 
of diagonal type have arithmetic root system; the conjecture is true in rank 2 and also in affine Cartan type \cite{aah-diag}.

\smallbreak
\item A class of braided vector spaces arising from abelian groups was introduced in \cite{aah-triang}; they are decomposable with components being points and blocks. 
Assuming the validity of the above Conjecture, the finite GK-dimensional Nichols algebras from this class were classified in \cite{aah-triang}.  
\end{itemize}

Beware that there are finite GK-dimensional Nichols algebras over abelian groups that do not belong to the referred class, see \cite[Appendix]{aah-triang}.

\smallbreak
The braided vector spaces in the class alluded to above can be labelled with flourished Dynkin diagrams.
The main result of \cite{aah-triang} says that the Nichols algebra of a braided vector space in the class has finite Gelfand-Kirillov dimension if and only if its flourished Dynkin diagram is admissible.

\emph{From now on we assume that  $p>2$}. (The case $p=2$ has to be treated separately).
In the present paper, we show,  adapting arguments from \cite{aah-triang}, 
that the  Nichols algebras of many braided vector spaces
of admissible flourished Dynkin diagrams are finite-dimensional. 
This result extends Example \ref{exa:qls} and the Jordan plane \cite{clw} and is reminiscent of 
a familiar phenomenon in Lie algebras in positive characteristic.
By bosonization we obtain many new examples of finite-dimensional pointed Hopf algebras. 

\subsection{The main result}
To describe more precisely our main Theorem we need first to discuss blocks. 

For $k<\ell \in\N_0$, we set $\I_{k,\ell}=\{k,k+1,\dots,\ell\}$, $\I_{\ell}=\I_{1, \ell}$.

\smallbreak
A \emph{block}  $\cV(\epsilon,\ell)$, where  $\epsilon\in \kut$ and $\ell \in \N_{\ge 2}$, is a braided vector space
with a basis $(x_i)_{i\in\I_\ell}$ such that for $i, j \in \I_\ell$, $1 < j$:
\begin{align}\label{equation:basis-block}
c(x_i \ot  x_1) &= \epsilon x_1 \ot  x_i,& c(x_i \ot  x_j) &=(\epsilon x_j+x_{j-1}) \ot  x_i.
\end{align}

In characteristic 0, the only Nichols algebras of blocks with finite $\GK$ are the Jordan plane $\toba(\cV(1,2))$ and the super Jordan plane
$\toba(\cV(-1,2))$; both have $\GK = 2$. 
In our context with $p>2$,  the Jordan plane $\toba(\cV(1,2))$ has dimension $p^2$ \cite{clw}; see 
Lemma \ref{prop:1block}.
Our starting result  is that the super Jordan plane $\toba(\cV(-1,2))$ has dimension $4p^2$, see
Proposition \ref{prop:-1block}.
For simplicity a block $\cV(\epsilon,2)$ of dimension 2 is called an $\epsilon$-block.
We also prove that a block $\cV(\epsilon,2)$ has finite-dimensional Nichols algebra only when $\epsilon = \pm 1$, 
see Proposition \ref{prop:exhaustion}.

\smallbreak
The braided vector spaces in this paper belong to the class analogous to the one considered in \cite{aah-triang}.
Briefly, $(V,c)$ belongs to this class if
\begin{align}\label{eq:bradinig-generalform}
V &=  V_{1} \oplus \dots \oplus V_t \oplus  V_{t + 1} \oplus \dots \oplus V_\theta,
\\ \label{eq:bradinig-generalform1}
c(V_i \otimes V_j) &=  V_j \otimes V_i,\, i,j\in \I_{\theta},
\end{align}
where  $V_h$ is a $\epsilon_h$-block, with $\epsilon_h^2 = 1$, for $h\in \I_t$; and $\dim V_{i} = 1$
with braiding determined by $q_{ii} \in \Bbbk^{\times}$ (we say that $i$ is a point),  $i\in \I_{t+1, \theta}$; the braiding between points $i$ and $j$ is given by $q_{ij} \in \Bbbk^{\times}$
while the braiding between a point and block, respectively two blocks, should have the form as in \eqref{eq:braiding-block-point}, respectively \eqref{eq:braiding-several-blocks-1pt}.
For convenience, we attach to $(V,c)$ a flourished graph $\D$ with $\theta$  vertices, those 
corresponding to a $1$-block decorated with  $\boxplus$, those to $-1$-block decorated with  $\boxminus$
and the point $i$ with $\overset{q_{ii}} {\circ}$.
If $i \neq j$ are points, and there is an edge between them  decorated by 
$\widetilde{q}_{ij} := q_{ij}q_{ji}$ when this is $\neq- 1$, or no edge if $\widetilde{q}_{ij} = 1$.
If $h$ is a block and $j$ is a point, then there is an edge between $h$ and $j$ 
decorated either by $\ghost_{hj}$ if the interaction is weak and $\ghost_{hj} \neq 0$ is the ghost, cf. \eqref{eq:discrete-ghost},
or by $(-, \ghost_{hj})$ if the interaction is mild and $\ghost_{hj}$ is the ghost; 
but no edge if the interaction is weak and $\ghost_{hj} = 0$.
There are no edges between blocks and we assume that the diagram is connected by
a well-known reduction argument. 

This class of braided vector spaces together with those of diagonal type does not exhaust
that of Yetter-Drinfeld modules arising from abelian groups; 
there are still those containing a pale block as in \cite[Chapter 8]{aah-triang}.
Synthetically our main result is the following.

\begin{theorem}\label{thm:main-intro} Let $V$ be a braided vector space as in the following list, then $\dim \NA (V) < \infty$.
\begin{enumerate}[leftmargin=*,label=\rm{(\alph*)}]
\item\label{item:block-point} $V$ has braiding \eqref{eq:braiding-block-point} and is listed in Table \ref{tab:toba-finitedim-block-point}, or

\item\label{item:block-points} $V$ has braiding
	\eqref{eq:braiding-block-several-point} and is listed in Table
	\ref{tab:toba-finitedim-block-points}, or

\item\label{item:blocks-point} $V$ has braiding \eqref{eq:braiding-several-blocks-1pt}, or

\item\label{item:paleblock-point} $V$ has braiding \eqref{eq:braiding-paleblock-point} and is listed in Table \ref{tab:toba-finitedim-paleblock-point}.
\end{enumerate}

By bosonization with suitable abelian groups, we get examples of finite-dimensional pointed Hopf algebras in positive characteristic. 
\end{theorem}

Concrete examples of such Hopf algebras are described in \S \ref{subsec:realizations-block}, \S \ref{subsec:realizations-block-pt}, \S \ref{subsec:realizations-block-pts}, \S \ref{subsec:realizations-blocks-pt} and \S \ref{subsec:realizations-paleblock-pt}. 
We also give a presentation by generators and relations of the Nichols algebras; references to this information and the dimensions
are also given in the Tables.

All braided vector spaces in this Theorem belong to 
the class described above except those in \ref{item:paleblock-point} that contain a pale block.

\begin{table}[ht]
\caption{{\small Finite-dimensional Nichols algebras of a block and a point}}\label{tab:toba-finitedim-block-point}
\begin{center}
\begin{tabular}{|c|c|c|c|c|c|c|}
\hline $V$ & {\scriptsize diagram}    & $q_{22}$  & $\ghost$ & $\NA(V)$  & {\small $\dim K$} 
& {\small $\dim \NA(V)$}  \\
\hline
$\lstr( 1, \ghost)$ & $\xymatrix{\boxplus \ar  @{-}[r]^{\ghost}  & \overset{1}{\bullet}}$  
& $1$  & \small{discrete}   & \S \ref{subsubsection:lstr-11disc} & $p^{\sa + 1}$   &    $p^{\sa + 3}$
\\\hline
$\lstr( -1, \ghost)$ & $\xymatrix{\boxplus \ar  @{-}[r]^{\ghost}  & \overset{-1}{\bullet}}$&  $-1$  & \small{discrete}  &  \S \ref{subsubsection:lstr-1-1disc} & $2^{\sa + 1}$   &   $2^{\sa + 1}p^{2}$
\\ \hline
$\lstr(\omega, 1)$& $\xymatrix{\boxplus \ar  @{-}[r]^{1}  & \overset{\omega}{\bullet}}$&
$\in \G'_3$  &  1 & \S \ref{subsubsection:lstr-1omega1} &   $3^3$ &   $3^3 p^{2}$
\\ \hline
$\lstr_{-}(1, \ghost)$
& $\xymatrix{\boxminus \ar  @{-}[r]^{\ghost}  & \overset{1}{\bullet}}$
& $1$  & \small{discrete} & \S \ref{subsubsection:lstr--11disc}  &  
$2^{\frac{\sa}{2}}p^{\frac{\sa}{2} + 1}$
& $2^{\frac{\sa}{2} + 2}p^{\frac{\sa}{2} + 3}$
\\ \hline
$\lstr_{-}(-1, \ghost)$
& $\xymatrix{\boxminus \ar  @{-}[r]^{\ghost}  & \overset{-1}{\bullet}}$
& $-1$  &  \small{discrete} &   \S \ref{subsubsection:lstr--1-1disc} 
&  $2^{\frac{\sa}{2} + 1}p^{\frac{\sa}{2}}$
&  $2^{\frac{\sa}{2} + 3}p^{\frac{\sa}{2} + 2}$
\\ \hline
$\cyc_1$&$\xymatrix{\boxminus \ar  @{-}[r]^{(-1, 1)}  &\overset{-1}{\bullet} }$ 
&  $-1$    &   1 & \S \ref{subsection:mild}  &  $16$  
&$64p^2$
\\\hline
\end{tabular}
\end{center}
\end{table}

\newcounter{Rowtable} 
\newcommand{\rowtable}[1]{\refstepcounter{Rowtable}\label{#1}}

\renewcommand{\arraystretch}{1.3}
\begin{table}[ht]
\caption{Finite-dimensional Nichols algebras of a block and several points,  $\omega \in \G'_3$. } \label{tab:toba-finitedim-block-points}
\begin{center}
\begin{tabular}{|c|c|c|c|}
\hline $V$ &  {\scriptsize diagram}    & $\NA(V)$     & {\scriptsize $dim \toba(V)$}   \\
\hline
 $\lstr(A_{\theta -1})$, &
$\xymatrix@C-10pt{\boxplus \ar  @{-}[r]^{1}  &\overset{-1}{\bullet} \ar  @{-}[r]^{-1}  & \overset{-1}{\circ}} \dots \xymatrix{
\overset{-1}{\circ} \ar  @{-}[r]^{-1}  & \overset{-1}{\circ} }$ &
\S \ref{subsubsection:lstr-a-n}      & $p^2 2^{6}$
\\ 
\cline{4-4}
$\theta > 2$&$\theta -1$ vertices &  & $p^2 2^{(\theta -1)(\theta-2)}$
\\ \hline
$\lstr(A_{2}, 2)$ &$\xymatrix{\boxplus \ar  @{-}[r]^{2}  &\overset{-1}{\bullet} \ar  @{-}[r]^{-1}  & \overset{-1}{\circ}}$    & \S \ref{subsubsection:lstr-a-22}  & $p^2 2^{12}$
\\ \hline
$\lstr(A(1\vert 0)_2; \omega)$ &$\xymatrix{\boxplus \ar  @{-}[r]^{1}  &\overset{-1}{\bullet} \ar  @{-}[r]^{\omega}  & \overset{-1}{\circ}}$ 
&  \S \ref{subsubsection:lstr-a(10)2}   & $p^2 2^73^4$
\\ \hline
$\lstr(A(1\vert 0)_1; \omega)$ &$\xymatrix{\boxplus \ar  @{-}[r]^{1}  &\overset{-1}{\bullet} \ar  @{-}[r]^{\omega^2 }  & \overset{\omega}{\circ}}$ 
& \S \ref{subsubsection:lstr-a(10)1} & $p^2 2^43^2$
\\ \hline
$\lstr(A(1\vert 0)_3; \omega)$ &$\xymatrix{\boxplus \ar  @{-}[r]^{1}  &\overset{\omega}{\bullet} \ar  @{-}[r]^{\omega^2 }  & \overset{-1}{\circ}}$  & \S \ref{subsubsection:lstr-a(10)3} & $p^2 2^73^4$
\\ \hline
$\lstr(A(1\vert 0)_1; r)$ &
$\xymatrix{\boxplus \ar  @{-}[r]^{1}  &\overset{-1}{\bullet} \ar  @{-}[r]^{r^{-1}}  & \overset{r}{\circ}}$, $r \in \G'_N, N > 3$   
&  \S \ref{subsubsection:lstr-a(10)1}  & $p^2 2^4N^2$
\\ \hline
$\lstr(A(2 \vert 0)_1; \omega)$ &
$\xymatrix{\boxplus \ar  @{-}[r]^{1}  &\overset{-1}{\bullet}\ar  @{-}[r]^{\omega}  & \overset{\omega^2}{\circ} \ar  @{-}[r]^{\omega}  & \overset{\omega^2}{\circ} }$ 
 & \S \ref{subsubsection:lstr-a(20)1} & $p^2 2^83^9$
\\ \hline
$\lstr(D(2 \vert 1); \omega)$ &
$\xymatrix{\boxplus \ar  @{-}[r]^{1}  &\overset{-1}{\bullet} \ar  @{-}[r]^{\omega}  & \overset{\omega^2}{\circ} \ar  @{-}[r]^{\omega^2}  & \overset{\omega}{\circ} }$ 
& \S \ref{subsubsection:lstr-D(21)} & $p^2 2^83^9$

\\ \hline
\end{tabular}
\end{center}

\end{table}

\begin{table}[ht]
\caption{{\small Finite-dimensional Nichols algebras of a pale block and a point.}}\label{tab:toba-finitedim-paleblock-point}
\begin{center}
\begin{tabular}{|c|c|c|c|c|c|c|}
\hline $V$ & $\epsilon$  & $\widetilde{q}_{12}$  & $q_{22}$ & $\NA(V)$  & {\small $\dim K$} 
& {\small $\dim \NA(V)$}  \\
\hline
$\eny_{p}(q)$ & $1$ & $1$  & $-1$  & 
\S \ref{subsubsec:paleblock-case1} & $2^{p}$   &    $2^p p^2$
\\\hline
$\eny_{+}(q)$ & $-1$ & $1$ & $1$ & 
\S \ref{subsubsec:paleblock-case-1} & $2p$   &   $2^3p$
\\ \hline
$\eny_{-}(q)$ & $-1$ & $1$ & $-1$ & 
\S \ref{subsubsec:paleblock-case-1} & $2p$ & $2^3p$
\\ \hline
$\eny_{\star}(q)$ & $-1$ & $-1$ & $-1$ & 
\S \ref{subsubsec:paleblock-case-1} & $2^4p^2$ & $2^6p^2$
\\ \hline
\end{tabular}
\end{center}
\end{table}

\subsection{Contents of the paper}\label{subsection:organization}
Section \ref{section:Preliminaries} is devoted to preliminaries. The next Sections contain the examples of finite-dimensional Nichols algebras and some realizations over abelian groups; each Section describes a family of braided vector spaces with a certain decomposition as we describe now. In Section \ref{sec:blocks} we compute Nichols algebras of a block. In Section \ref{sec:yd-dim3} we present examples of Nichols algebras corresponding to one block and one point, while in Section \ref{sec:yd-dim>3} we consider the case one block and several points. Section \ref{section:YD>3-severalblocks-1pt-poseidon} is devoted to examples of several blocks and one point. Finally in  Section \ref{sec:paleblock}
we give examples of finite-dimensional Nichols algebras whose braided vector
spaces decompose as one pale block and one point.  
\section{Preliminaries}\label{section:Preliminaries}

\subsection{Conventions}\label{subsection:conventions}
The $\qmb$-numbers are the polynomials
\begin{align*}
(n)_\qmb &=\sum_{j=0}^{n-1}\qmb^{j}, & (n)_\qmb^!&=\prod_{j=1}^{n} (j)_\qmb, &
\binom{n}{i}_\qmb & =\frac{(n)_\qmb^!}{(n-i)_\qmb^!(i)_\qmb^!} \in \Z[\qmb],
\end{align*}
$n\in \N$, $0 \leq i \leq n$. If  $q\in\ku$, then $(n)_q$, $(n)_q^!$, $\binom{n}{i}_q$ 
denote the  evaluations of $(n)_\qmb$, $(n)_\qmb^!$, $\binom{n}{i}_\qmb$ at $\qmb = q$.

Let $\G_N$ be the group of $N$-th roots of unity, and $\G_N'$ the subset of primitive roots of order $N$;
$\G_{\infty} = \bigcup_{N\in \N} \G_N$.
All the vector spaces, algebras and tensor products  are over $\ku$.

All Hopf algebras have bijective antipode.

\subsection{Yetter-Drinfeld modules}\label{subsection:yd} 
Let $\Gamma$ be an abelian group. We denote by $\widehat \Gamma$ the group of  characters  of $\Gamma$.
The category $\ydG$ of Yetter-Drinfeld modules over the group algebra $\ku \Gamma$ 
consists of $\Gamma$-graded $\Gamma$-modules,
the $\Gamma$-grading being denoted by $V = \oplus_{g\in \Gamma} V_g$; that is, $hV_g=V_g$ for all $g,h\in \Gamma$.
If $g\in \Gamma$ and $\chi \in \widehat\Gamma$, then  the one-dimensional vector space $\ku_g^{\chi}$,
with action and coaction given by $g$ and $\chi$, is in $\ydh$.
Let  $W \in \ydG$ and $(w_i)_{i\in I}$ a basis of $W$ consisting of homogeneous elements of degree $g_i$, $i\in I$,
respectively. Then there are skew derivations $\partial_i$, $i\in I$, of $T(W)$ such that
for all $x,y\in T(W)$, $i, j \in I$
\begin{align}\label{eq:skewderivations}
\partial _i(w_j) & =\delta _{ij},&
\partial _i(xy) &= \partial _i(x)(g_i\cdot y)+x\partial _i(y).
\end{align}

For a definition of Yetter-Drinfeld modules over arbitrary Hopf algebras we
refer e.g. to \cite[11.6]{Ra}. 

\subsection{Nichols algebras}\label{subsection:nichols}
Nichols algebras are graded Hopf algebras $\toba = \oplus_{n\ge 0} \toba^n$ in $\ydh$ coradically graded and generated
in degree one. 
They are completely determined by 
$V := \toba^1 \in \ydh$
and it is customary to denote $\toba = \toba(V)$.
If $W\in \ydG $ as in Subsection~\ref{subsection:yd}, then
	the skew-derivations $\partial_i$ induce skew-derivations on $\toba(W)$.
	Moreover,	an element $w\in \toba^k(W)$, $k\ge 1$, is zero if and only if
$\partial_i(w)=0$ in $\toba(W)$ for all $i\in I$.
A pre-Nichols algebra of $V$ is a graded Hopf algebra in $\ydh$ generated in degree one, with the one-component
 isomorphic to $V$.

\begin{example}\label{exa:dim1}
	Let $V$ be of dimension 1 with braiding $c = \epsilon \id$. 
	Let $N$ be the smallest natural number such that $(N)_{\epsilon} = 0$.
	Then $\toba(V) = \ku [T]/\langle T^N\rangle$, or $\toba(V) = \ku [T]$ if such
	$N$ does not exist.
\end{example}

\medbreak
A braided vector space $V$ is of diagonal type if there exist a basis
$(x_i)_{i\in \I_\theta }$ of $V$ and $\bq = (q_{ij})_{i,j\in \I_\theta }\in \Bbbk^{\theta \times \theta }$
such that $q_{ij}\ne 0$ and $c(x_i\otimes x_j)=q_{ij}x_j\otimes x_i$ for all $i,j\in \I = \I_\theta $.
Given a braided vector space
$V$
of diagonal type  with a basis $(x_i)$, we denote in $T(V)$, or $\NA(V)$, or any intermediate Hopf algebra,
\begin{align}\label{eq:xijk}
x_{ij} &= (\ad_c x_i)\, x_j,& x_{i_1i_2 \dots i_M} &= (\ad_c x_{i_1})\, x_{i_2 \dots i_M},
\end{align}
for $i,j,i_1, \dots, i_M \in \I$, $M\ge 2$.
A braided vector space $V$ of diagonal type is of Cartan type if there exists 
a generalized Cartan matrix $\ba = (a_{ij})$ such that $q_{ij}q_{ji} = q_{ii}^{a_{ij}}$ for all $i\neq j$.

\begin{theorem}\label{thm:nichols-diagonal-finite-gkd} 
	If $V$ is of Cartan type with matrix $\ba$ that is not finite, then $\dim  \NA(V) = \infty$.
\end{theorem}

\pf The argument in \cite[Proposition 3.1]{aah-diag} is characteristic-free and applies here
because there are infinite real roots in the root system of $\ba$.
\epf

\section{Blocks}\label{sec:blocks}
We consider braided vector spaces $\cV(\epsilon,2)$  with braiding 
\eqref{equation:basis-block}, $\epsilon^2 = 1$. 

\subsection{The Jordan plane}\label{subsection:jordanian}

Here we deal with  $\cV=\cV(1,2)$. In characteristic 0, $\toba(\cV)$ is the well-known algebra presented by 
$x_1$ and $x_2$ with the relation \eqref{eq:rels B(V(1,2))}.
In positive characteristic, $\toba(\cV)$ is a truncated version of that algebra.  

\begin{lemma} \label{prop:1block} \cite{clw}
$\toba(\cV)$ is presented by generators $x_1,x_2$ and relations
\begin{align}\label{eq:rels B(V(1,2))}
&x_2x_1-x_1x_2+\frac{1}{2}x_1^2,
\\\label{eq:rels B(V(1,2))-x1p}
&x_1^p,
\\
\label{eq:rels B(V(1,2))-x2p}
&x_2^p.
\end{align}
Also $\dim \toba(\cV) = p^2$ and $\{ x_1^a x_2^b: 0 \le a,b< p\}$ is a basis of $\toba(\cV)$. \qed
\end{lemma}
In characteristic 2, the relations of $\toba(\cV)$ are different.  

\smallbreak
Let $\Gamma = \Z/p = \langle g \rangle$.
We realize  $\cV$ in $\ydG$ by 
$g\cdot x_1 = x_1$, $g\cdot x_2 = x_2 + x_1$, $\deg x_i = g$, $i\in \I_2$.
Thus  the Hopf algebra $\toba(\cV) \# \ku \Gamma$ has dimension $p^3$.

\subsection{The  super Jordan plane}\label{subsection:super-jordanian}
Let $\cV = \cV(-1,2)$ be the braided vector space with basis $x_1$, $x_2$ and braiding
\begin{align}\label{eq:bloque2}
c(x_i \otimes x_1) &= - x_1 \otimes x_i,& c(x_i \otimes x_2) &= (- x_2 + x_1) \otimes x_i,& i&\in \I_2.
\end{align}
Let $g$ be a generator of the cyclic group $\Z$.
We realize  $\cV(-1,2)$ in $\ydz$ by $g\cdot x_1 = -x_1$, $g\cdot x_2 = - x_2 + x_1$, $\deg x_i = g$, $i\in \I_2$.  As in \eqref{eq:xijk}, 
\begin{align}\label{eq:x12}
x_{21} &= (\ad_c x_2)\, x_1 = x_2x_1 + x_1 x_2.
\end{align}
The Nichols algebra $\toba(\cV(-1,2)) = T(\cV(-1,2)) / \cJ(\cV(-1,2))$ (called the super Jordan plane) was studied in \cite[3.3]{aah-triang} over fields of characteristic 0. 
Assuming $p > 2$, the basic features of $\toba(\cV(-1,2))$ are summarized here:

\begin{prop} \label{prop:-1block} The defining ideal
$\cJ(\cV(-1,2))$ is generated by
\begin{align}\label{eq:rels-B(V(-1,2))-1}
&x_1^2, \\
\label{eq:rels-B(V(-1,2))-2}
&x_2x_{21}- x_{21}x_2 - x_1x_{21},
\\ \label{eq:rels-B(V(-1,2))-4} &x_{21}^p,
\\\label{eq:rels-B(V(-1,2))-3} &x_2^{2p}.
\end{align}
The set $B = \{x_1^a x_{21}^bx_2^c: a\in\I_{0,1},
b\in\I_{0, p-1}, c\in\I_{0, 2p-1}\}$ is a basis of $\toba(\cV)$ and $\dim \toba(\cV) = 4p^2$.
\end{prop}

\pf 
Since $\partial_2(x_{21})=0$, we have that  $\partial_2(x_{21}^n)=0$ for every $n\in \N$. Also
$\partial_1(x_{21})=x_1$ and $g\cdot x_{21} = x_{21}$.
Both \eqref{eq:rels-B(V(-1,2))-1} and \eqref{eq:rels-B(V(-1,2))-2} are $0$ in $\toba(\cV)$
being annihilated by $\partial_1$ and $\partial_2$, cf. \eqref{eq:skewderivations}.
From \eqref{eq:rels-B(V(-1,2))-1} and \eqref{eq:rels-B(V(-1,2))-2} we see that in $\toba(\cV)$
\begin{align}\label{eq:rels-B(V(-1,2))-5}
	x_2^2 x_1 &= x_1 (x_2^2 + x_{21}), \\
\label{eq:rels-B(V(-1,2))-dos}
x_{21}x_1 &= x_1x_2x_1 = x_1x_{21}.
\end{align}
By the preceding, we have
\begin{align*}
\partial_1(x_{21}^n) &= \sum_{1\le i \le n} x_{21}^{i-1} \partial_1(x_{21}) x_{21}^{n-i} = n x_1x_{21}^{n-1}.
\end{align*}
Hence $x_{21}^p = 0$, i.e. \eqref{eq:rels-B(V(-1,2))-4} holds. 
Next we prove \eqref{eq:rels-B(V(-1,2))-3}.
Clearly $\partial_1(x_{2}^n) = 0$ for every $n\in \N$. We observe that 
\begin{align*}
g\cdot x_{2}^2 &= (- x_2 + x_1)^2 = x_{2}^2 - x_{21}, & \partial_2(x_{2}^2) &= g\cdot x_2 + x_2 = x_1.
\end{align*}
Setting for simplicity $a := x_{21}$ and $b := x_2^2$, we have for any $n \in \N$:
\begin{align*}
\partial_2(x_{2}^{2n}) &= \sum_{1\le i \le n} x_{2}^{2(i-1)} \partial_2(x_{2}^2) \left(g\cdot x_{2}^{2(n-i)}\right) = 
\sum_{1\le i \le n} b^{i-1} x_1 \left(b -a\right)^{n-i}.
\end{align*}
By  \eqref{eq:rels-B(V(-1,2))-2}, \eqref{eq:rels-B(V(-1,2))-5} and \eqref{eq:rels-B(V(-1,2))-dos} we have
\begin{align} \label{eq:a-b-conditions}
ax_{1} &= x_{1}a, & b x_{1} &= x_1(b + a), & 
b a &= a(a + b), 
\end{align}
hence
\begin{align*}
\partial_2(x_{2}^{2n}) &= x_1\sum_{1\le i \le n} \left(b + a\right)^{i-1}  \left(b -a\right)^{n-i}.
\end{align*}
We prove recursively that for all $n \in \N$
\begin{align}\label{eq:a-b-formulae}
\left(b -a\right)^{n} &= b^n-n ab^{n-1}, &
A_n := \sum_{i \in \I_n} \left(b + a\right)^{i-1}  \left(b -a\right)^{n-i} &= n b^{n-1}.
\end{align}
The case $n = 1$ is evident.
We start with the first identity:
\begin{align*}
\left(b -a\right)^{n+1} &= (b-a)\left(b -a\right)^{n} = (b-a)(b^n-n ab^{n-1})
\\ & = b^{n+1} -n bab^{n-1} - ab^n + n a^2b^{n-1}  = b^{n+1} -(n + 1)ab^{n} 
\end{align*}
as desired.  For the second identity we use the first:
\begin{align*}
A_{n + 1} & =  \left(b -a\right)^{n} + (b+ a)A_{n} = b^n-n ab^{n-1} + (b+ a) n b^{n-1}
=(n+1) b^n.
\end{align*}
The claim is proved; summarizing we have
\begin{align}\label{eq:rels-B(V(-1,2))-evenpower}
\partial_2(x_{2}^{2n}) &= n x_1 x_2^{2(n-1)}.
\end{align}
In particular, this implies \eqref{eq:rels-B(V(-1,2))-3}.

\medbreak
We now argue as in \cite[3.3.1]{aah-triang}.
The quotient $\cBt$ of $T(\cV)$ by \eqref{eq:rels-B(V(-1,2))-1}, \eqref{eq:rels-B(V(-1,2))-2}, \eqref{eq:rels-B(V(-1,2))-4} and \eqref{eq:rels-B(V(-1,2))-3} projects onto $\toba(\cV)$ and
the subspace $I$ spanned by $B$ is a left ideal of $\cBt$, by  \eqref{eq:rels-B(V(-1,2))-2}, \eqref{eq:rels-B(V(-1,2))-dos}.
Since $1\in I$, $\cBt$ is spanned by $B$. To prove that $\cBt \simeq \toba(\cV)$, we just need to show that
$B$ is linearly independent in $\toba(\cV)$. We claim that this is equivalent to prove that $B'=\{x_2^c x_{21}^b x_1^a: a\in\{0,1\}, b\in\I_{0, p-1}, c\in\I_{0, 2p-1}\}$ is linearly independent.
Indeed, $\cBt$ is spanned by $B'$ since the subspace spanned $B'$ is also a left ideal; if $B'$ is linearly independent, then
the dimension of $\cBt$ is $4p^2$, so $B$ should be linearly independent and vice versa.
Suppose that there is a non-trivial linear combination
of elements of $B'$ in $\toba(\cV)$ of minimal degree. As
\begin{align} \label{eq:-1del1}
\partial_1(x_2^c x_{21}^b)&= b \,  x_2^c x_{21}^{b-1} x_1, & \partial_1(x_2^c x_{21}^b x_1)&= x_2^c x_{21}^b,
\end{align}
such linear combination does not have terms with $a$ or $b$ greater than 0. We claim that the elements 
$x_2^c$, $c\in\I_{0, 2p-1}$,
are linearly independent, yielding a contradiction.
By homogeneity it is enough to prove that they are $\neq 0$. If $c$ is even this follows from \eqref{eq:rels-B(V(-1,2))-evenpower}.
If $c =2n +1$ with $n < p$, then
\begin{align*}
\partial_2(x_{2}^{2n + 1}) &= \partial_2(x_{2}^{2n}) g\cdot x_2 +  x_{2}^{2n} = 
-n x_1 x_2^{2n-1} + x_{2}^{2n}.
\end{align*}
Again a degree argument gives the desired claim.
\epf

Let $\Gamma = \Z/2p$. We may realize  $\cV(-1,2)$ in $\ydG$ by the same formulas as above; thus 
$\toba(\cV) \# \ku \Gamma$ is a pointed Hopf algebra of dimension $8p^3$.

\subsection{Realizations}\label{subsec:realizations-block}

Let $H$ be a Hopf algebra. 
A \emph{YD-pair} for $H$  is a pair
$(g, \chi) \in G(H) \times \Alg(H, \ku)$ such that
\begin{align}\label{eq:yd-pair}
\chi(h)\,g  &= \chi(h\_{2}) h\_{1}\, g\, \Ss(h\_{3}),& h&\in H.
\end{align}
Let $\ku_g^{\chi}$ be a one-dimensional vector space with $H$-action and $H$-coaction  given by $\chi$ and $g$
respectively; then \eqref{eq:yd-pair} says that  $\ku_g^{\chi} \in \ydh$.

If $\chi\in \Alg(H, \ku)$, then the space of $(\chi, \chi)$-deri\-vations is
\begin{align*}
\Der_{\chi,\chi}(H, \ku) &= \{\eta\in H^*: \eta(h\ell) = \chi(h)\eta(\ell) + \chi(\ell)\eta(h) \, \forall h,\ell \in H\}.
\end{align*}

A \emph{YD-triple} for $H$ is a collection $(g, \chi, \eta)$ where
$(g, \chi)$, is a YD-pair for $H$, $\eta \in \Der_{\chi,\chi}(H, \ku)$, $\eta(g) = 1$ and
\begin{align}\label{eq:YD-triple}
\eta(h) g_1 &= \eta(h\_2) h\_1 g_2 \Ss(h\_3), & h&\in H.
\end{align}
Given a YD-triple $(g, \chi, \eta)$ we define $\cV_g(\chi,\eta) \in \ydh$ as the
 vector space with a basis $(x_i)_{i\in\I_2}$, whose $H$-action and $H$-coaction are given by
\begin{align*}
h\cdot x_1 &= \chi(h) x_1,& h\cdot x_2&=\chi(h) x_2 + \eta(h)x_{1},& \delta(x_i) &= g\otimes x_i,&
h&\in H, \, i\in \I_2;
\end{align*}
the compatibility is granted by \eqref{eq:yd-pair}, \eqref{eq:YD-triple}.
As a braided vector space, $\cV_g(\chi, \eta)\simeq \cV(\epsilon,2)$,
$\epsilon = \chi(g)$. 

Consequently, if $H$ is finite-dimensional and $\epsilon^2 = 1$,
then $\toba (\cV_g(\chi, \eta)) \# H$ is a Hopf algebra satisfying
\begin{align}
\dim \big(\toba (\cV_g(\chi, \eta)) \# H\big) &= \begin{cases}
&p^2\dim H, \text{ when } \epsilon = 1,  \\
&4p^2 \dim H, \text{ when } \epsilon = -1.
\end{cases}
\end{align}

\subsection{Exhaustion in rank $2$}\label{subsection:exhaustion-rank2}

We recall some facts from \cite[\S 3.4]{aah-triang}.

\smallbreak
Let $H$ be a Hopf algebra with bijective antipode and $V\in \ydh$.
Let  $0 =V_0 \subsetneq V_1 \dots \subsetneq V_d = V$ be a flag of Yetter-Drinfeld submodules 
with $\dim V_i=\dim V_{i-1}+1$ for all $i$.
Then $V^{\textrm{diag}} := \gr V$ is of diagonal type. If $\toba$ is a pre-Nichols algebra of $V$, then
it is a graded filtered Hopf in $\ydh$ and
$\Bdiag := \gr \toba$ is a pre-Nichols algebra of  $V^{\textrm{diag}}$.

\begin{prop} \label{prop:exhaustion} 
Let $\epsilon\in\Bbbk^{\times}$.  If $\dim \toba(\cV(\epsilon, 2)) < \infty$, then $\epsilon^2 =1$.
\end{prop}

\pf
Let $\cV = \cV(\epsilon, 2)$; it has a flag as above 
and $\Vdiag$ is the braided vector space of diagonal type
with matrix $(q_{ij})_{i,j\in\I_2}$, $q_{ij}=\epsilon$ for all $i,j\in\I_2$. Hence
\begin{align}\label{eq:gr br Hopf algebra}
\dim \NA(\Vdiag) \leq \dim \NA(\cV(\epsilon,2)).
\end{align}

\begin{paso} \label{prop:V-lambda-generic}
	If $\epsilon\notin \G_{\infty}$, then $\dim \toba(\cV(\epsilon,2)) = \infty$.
\end{paso}

\pf
Here $\dim  \toba(\Vdiag) = \infty$ by Example \ref{exa:dim1} and \eqref{eq:gr br Hopf algebra} applies.
\epf

\begin{paso} \label{prop:V-lambda-N>3}
	If  $\epsilon\in\G'_N$, $N\geq 4$, then $\dim  \toba(\cV(\epsilon,\ell)) = \infty$ for all $\ell\ge2$.
\end{paso}

\pf
Here $\Vdiag$ is of Cartan type with Cartan matrix
$\begin{pmatrix} 2 & 2-N \\ 2-N & 2\end{pmatrix}$.
Thus Theorem \ref{thm:nichols-diagonal-finite-gkd} and \eqref{eq:gr br Hopf algebra} apply.
\epf

\begin{paso} \label{le:infGK3}
	Let $\epsilon\in\G'_3$. Then $\dim  \toba(\cV(\epsilon,2)) = \infty$.
\end{paso}
\pf The proof of \cite[\S 3.5 -- Step 3]{aah-triang} holds verbatim.
\epf
The Proposition is proved. \epf

\section{One block and one point}\label{sec:yd-dim3}
\subsection{The setting and the statement}\label{subsection:YD3-setting}

Let $(q_{ij})_{1 \le i,j \le 2}$ be a matrix of invertible scalars and $a \in \ku$.
We assume that $\epsilon :=q_{11}$ satisfies $\epsilon^2 = 1$.
Let $V$ be a braided vector space  with a basis $(x_i)_{i\in\I_3}$ and a braiding given by
\begin{align}\label{eq:braiding-block-point}
(c(x_i \otimes x_j))_{i,j\in \I_3} &=
\begin{pmatrix}
\epsilon x_1 \otimes x_1&  (\epsilon x_2 + x_1) \otimes x_1& q_{12} x_3  \otimes x_1
\\
\epsilon x_1 \otimes x_2 & (\epsilon x_2 + x_1) \otimes x_2& q_{12} x_3  \otimes x_2
\\
q_{21} x_1 \otimes x_3 &  q_{21}(x_2 + a x_1) \otimes x_3& q_{22} x_3  \otimes x_3
\end{pmatrix}.
\end{align}
Let $V_1 = \langle x_1, x_2\rangle$ (the block) and $V_2 = \langle x_3 \rangle$ (the point). 
Let $\Gamma = \Z^2$ with canonical basis $g_1, g_2$.
We realize $(V, c)$ in $\ydG$ as $\cV_{g_1}(\chi_1, \eta) \oplus \ku_{g_2}^{\chi_2}$
with suitable  $\chi_1, \chi_2$ and $\eta$, where $V_1 = \cV_{g_1}(\chi_1, \eta)$, while $V_2 = \ku_{g_2}^{\chi_2}$.
Thus $V_1 \simeq \cV(\epsilon, 2)$;  thus we use the notations and results from \S \ref{subsection:super-jordanian}.

The \emph{interaction} between the block and the point is $q_{12}q_{21}$; it is
\begin{align*}
\text{weak if } q_{12}q_{21}&= 1, &\text{mild if } q_{12}q_{21}&= -1,&\text{strong if } q_{12}q_{21}&\notin \{\pm 1\}.
\end{align*}

In characteristic 0, we introduced a normalized version of $a$
called the ghost, which is discrete when it belongs to $\N$. 
In our context, $p>2$, we need a variant of this notion.
First we say that $V$ has \emph{discrete ghost} if $a \in \Fp^{\times}$.
When this is the case, we pick a representative $\sa \in \Z$ of $2a$ by imposing
\begin{align}\label{eq:discrete-ghost}
\sa &\in \begin{cases}
\{1- p, \dots, -1 \}, &  \epsilon = 1, \\
\{1, \dots, 2p-1 \} \cap 2\Z, &  \epsilon = -1;
\end{cases}
& \text{ set }
\ghost := \begin{cases} -\sa, &\epsilon = 1, \\
\sa, &\epsilon = -1.
\end{cases}
\end{align}
Then $\ghost$ is called the \emph{ghost}.
In this Section we consider the following braided vectors spaces with braiding \eqref{eq:braiding-block-point},
where the ghost is discrete and $q_{22}\in\G_{\infty}$:
\begin{align*}
&\lstr(q_{22}, \ghost):& &\text{weak interaction, }& &\epsilon = 1;
\\
&\lstr_{-}(q_{22}, \ghost):& &\text{weak interaction, } & &\epsilon = -1;
\\
&\cyc_1:& &\text{mild interaction,  }& &\epsilon = q_{22} = -1, \quad \ghost = 1. 
\end{align*}

In this Section, we shall prove part \ref{item:block-point} of Theorem \ref{thm:main-intro}.

\begin{theorem}\label{thm:point-block} Let $V$ be a braided vector space with braiding \eqref{eq:braiding-block-point}.
If $V$ is as in Table \ref{tab:toba-finitedim-block-point}, then $\dim \NA (V) < \infty$.
\end{theorem}

To prove the Theorem, we consider
$K=\NA (V)^{\mathrm{co}\,\NA (V_1)}$.
By \cite[Proposition 8.6]{HS}, 
$\NA(V) \simeq K \# \NA (V_1)$ and $K$ is the Nichols algebra of
\begin{align} \label{eq:1bpK^1}
K^1= \ad_c\NA (V_1) (V_2).
\end{align}
Now $K^1\in {}^{\NA (V_1)\# \ku \Gamma}_{\NA (V_1)\# \ku \Gamma}\mathcal{YD}$ with the adjoint action
and the coaction given by
\begin{align} \label{eq:coaction-K^1}
\delta &=(\pi _{\NA (V_1)\#  \ku \Gamma}\otimes \id)\Delta _{\NA (V)\#  \ku \Gamma}.
\end{align}

In order to describe $K^1$, we set
\begin{align}\label{eq:zn}
z_n &:= (\ad_c x_2)^n x_3,& n&\in\N_0.
\end{align}

\subsection{Weak interaction}\label{subsection:weak}
\emph{Here $q_{12}q_{21} = 1$}. In general, 
\begin{align}
c^2_{\vert V_1 \otimes V_2} = \id \iff q_{12}q_{21} = 1 \text{ and } a=0.
\end{align}
If $a = 0$, then
\begin{align}
\label{eq:block-point-disconnected}
\NA (V) \simeq \NA(\cV(\epsilon, 2)) \underline{\otimes} \NA(\ku x_3).
\end{align}
Here $\underline{\otimes}$ denotes the \emph{braided} tensor product of Hopf algebras (the structure of Hopf algebra in $\ydg$)

\emph{From now on we assume that the ghost is discrete,} in particular $\neq 0$.
We follow the exposition in \cite[\S 4.2]{aah-triang}.

\begin{lemma} \label{le:-1bpz} The following formulae hold in $\NA(V)$ for all $n\in\N_0$:
\begin{align}
\label{eq:-1block+point}
\begin{aligned}
g_1\cdot z_n &= \epsilon^nq_{12}z_n,& x_1z_n &= \epsilon^n q_{12}z_nx_1,& x_{21}^n x_2 &= (n\epsilon x_1 + x_2) x^n_{21},
\\
g_2\cdot z_n &= q_{21}^nq_{22}z_n,& x_{21}z_n &= q_{12}^2z_nx_{21},& 
x_2z_n &= \epsilon^nq_{12}z_nx_2 + z_{n+1}. 
\end{aligned}
\end{align}
\end{lemma}

\pf The proof of \cite[Lemma 4.2.1]{aah-triang} is valid in any characteristic.
\epf

Let $(\mu_n)_{n \in\N_0}$ be the family of elements of $\ku$ defined recursively by
\begin{align*}
\mu_0 &= 1, & \mu_{2k + 1} &=- (a + k\epsilon) \mu_{2k},& \mu_{2k} &= (a + k +  \epsilon (a + k - 1))\mu_{2k - 1}.
\end{align*}
This can be reformulated as
\begin{align}\label{eq:def-mu-n}
\mu_{n + 1} &= \begin{cases} (2a + n)\mu_n &\text{if $n$ is odd,} \\
-\frac{2a + n}{2}  \mu_n &\text{if $n$ is even,}
\end{cases}   &\text{when } \epsilon &= 1;
\\
\mu_{n + 1} &= \begin{cases} \mu_n &\text{if $n$ is odd,} \\
-(a - \frac n{2}) \mu_n &\text{if $n$ is even,}  
\end{cases}  &\text{when } \epsilon &= -1.
\end{align}
Thus  $\mu_n =0 \iff n > \vert \sa\vert$.

\begin{lemma}\label{lemma:derivations-zn}
For all $k\in \{1, \dots, p-1 \}$, $\partial_1(z_k) =\partial_2(z_k) =  0$,
\begin{align}\label{eq:derivations-zn}
\partial_3(z_{2k}) &= \mu_{2k} x_{21}^k,&  \partial_3(z_{2k+1}) &= \mu_{2k + 1} x_1x_{21}^k.
\end{align}
Therefore, $z_n =0 \iff n > \vert \sa\vert$.
\end{lemma}
\pf
The proof of \cite[Lemma 4.2.2]{aah-triang} is valid in any characteristic, 
taking into the account the conventions on $\sa$.
\epf

If $\epsilon = 1$, then \eqref{eq:derivations-zn} says that
\begin{align}\label{eq:derivations-zn-eps1}
\partial_3(z_{2k}) &= \frac{(-1)^k}{2^k} \mu_{2k} x_{1}^{2k}, &
\partial_3(z_{2k+1}) &= \frac{(-1)^k}{2^k}\mu_{2k + 1} x_{1}^{2k + 1}.
\end{align}

\bigbreak
Recall that
$K=\NA (V)^{\mathrm{co}\,\NA (V_1)} \simeq \toba(K^1)$, $K^1= \ad\NA (V_1)(V_2)$.

\begin{lemma} \label{lemma:K-basis} The family $(z_n)_{0\le n\le \vert \sa\vert}$ is a basis of $K^1$.
\end{lemma}

\pf
The family is linearly independent,
because the $z_n$'s are homogeneous of distinct degrees, and  are $\neq 0$ by \eqref{eq:derivations-zn}.
We have for all $n\in\N_0$
\begin{align}\label{eq:adx1-zn}
\ad_c x_1 (z_n) =  x_1z_n - g_1\cdot z_n x_1 \overset{\eqref{eq:-1block+point}}= \epsilon^n q_{12}z_nx_1 -  \epsilon^nq_{12}z_nx_1 = 0,
\\
\label{eq:adx12-zn}
\ad_c x_{21} (z_n) = \ad_c  x_2 \ad_c x_1 (z_n) - \epsilon \ad_c x_1 \ad_c x_2 (z_n) = 0.
\end{align}
The Lemma follows.
\epf

\smallbreak
If $\epsilon = 1$, then we define recursively
$\nu_{k,n}$  as follows:  $\nu_{n,n}=1$,
\begin{align*}
\nu_{0, n+1} &= -  \big(\frac{n}{2}+ a \big)\nu_{0,n}, &
\nu_{k,n+1}&=\nu_{k-1,n}-\left(\frac{n+k}{2}+ a\right)\nu_{k,n}, & &  1\le k\le n.
\end{align*}

\begin{lemma} \label{le:zcoact}
The coaction \eqref{eq:coaction-K^1} on $z_n$, $0\le n\le \vert \sa\vert$, 
is given by \eqref{eq:coact-zn}, when $\epsilon = 1$,
and by \eqref{eq:coact-zn-even}, \eqref{eq:coact-zn-odd}, when $\epsilon = -1$:

\begin{align} \label{eq:coact-zn}
\delta (z_{n}) &= \sum _{k=0}^n \nu_{k,n}\, x_1^{n-k}g_1^{k} g_2 \otimes z_k.
\end{align} \begin{align}\label{eq:coact-zn-even}
\delta (z_{2n}) &=
\sum _{k=1}^n k\binom n k \mu_{k,n} \,
x_1x_{21}^{n-k}g_1^{2k-1}g_2\otimes z_{2k-1}
\\ \notag&\qquad  + \sum _{k=0}^n \binom n k\mu_{k, n} x_{21}^{n-k}g_1^{2k}g_2\otimes z_{2k},
\\ \label{eq:coact-zn-odd}
\delta (z_{2n+1}) &=
\sum _{k=0}^n \binom nk \mu_{k, n+1} \,   x_1x_{21}^{n-k}g_1^{2k}g_2\otimes z_{2k}
\\ \notag& \qquad + \sum _{k=0}^n \binom nk \mu_{k+1, n+1} x_{21}^{n-k}g_1^{2k+1}g_2\otimes z_{2k+1}.
\end{align}
\end{lemma}

\begin{proof} The proof of \cite[Lemma 4.2.1]{aah-triang} can be adapted since $\binom nk \neq 0$  by assumption.
\end{proof}

\smallbreak
We are ready to prove the finite-dimensionality in the case of weak interaction.
We claim that the braided vector space  $K^1$ is of diagonal type
with  braiding matrix
\begin{align*}
(\bp_{ij})_{0\le i,j\le \vert \sa \vert} &= (\epsilon^{ij}q_{12}^iq_{21}^jq_{22})_{0\le i,j\le \vert \sa \vert}.
\end{align*}
Hence, the corresponding generalized  Dynkin diagram  has labels
\begin{align*}
\bp_{ii}&= \epsilon^iq_{22},& \bp_{ij}\bp_{ji}&= q_{22}^2, & &i\neq j\in \I_{0, \vert \sa \vert}.
\end{align*}

Indeed, by Lemma \ref{lemma:K-basis} it is enough to compute
\begin{align*}
c(z_i \otimes z_j) &= g_1^ig_2 \cdot z_j \otimes z_i = \epsilon^{ij}q_{12}^iq_{21}^jq_{22} z_j \otimes z_i,
\end{align*}
by Lemmas \ref{le:zcoact} and \ref{le:-1bpz}, \eqref{eq:adx1-zn} and \eqref{eq:adx12-zn}. We proceed then case by case.

\begin{case} \label{case:1}
$q_{22}^2=1$.
\end{case}
Here the Dynkin diagram of $K^1$ is totally disconnected with
vertices $i\in \I_{0, \vert \sa \vert}$ labelled with $\epsilon^iq_{22}$.
The vertices with label $1$, respectively $-1$, contribute with $p$, respectively $2$, to $\dim \NA (K^1)$.

\begin{case} \label{case:2}
$\epsilon = 1$, $q_{22}\in\G_3'$, $\vert \sa \vert = 1$ (provided that $p >3$).
\end{case}
The Dynkin diagram is of Cartan type $A_2$, so $\dim \NA (K^1) < \infty$.

\subsection{The presentation by generators and relations}\label{subsection:point-block-presentation}
We still assume that the interaction is weak. We start by some general Remarks that are proved exactly as in \cite[\S 4.3]{aah-triang}.

\begin{remark}\label{rem:xk qcommutes with z_k}
 Let
\begin{align*}
y_{2k} &=x_{21}^k, &  y_{2k+1} &= x_1x_{21}^k, & k & \in \N_0
\end{align*}
By Lemma \ref{le:-1bpz}
\begin{align}\label{eq:yn zt commute}
\partial_3(z_t)&=\mu_t y_t, & z_t y_n &= \epsilon^{nt}q_{21}^n y_n z_t, & &t, n \in\N_0.
\end{align}
\end{remark}

\begin{lemma}\label{lemma:relations L(pm1,pm1,G)}
Assume that $\epsilon^2=q_{22}^2=1$. In $\toba(\lstr(q_{22}, \ghost))$, or correspondingly
$\toba_{-}(\lstr(q_{22}, \ghost))$,
\begin{align}
\label{eq:q-serre}
z_{|\sa|+1}&=0, \\
z_t z_{t+1} &= q_{21}q_{22} z_{t+1} z_t & t\in\N_0,& \, t<|\sa|, \label{eq:zt zt+1 qcommute}\\
z_t^2&=0 & t\in\N_0&, \,  \epsilon^tq_{22}=-1. \label{eq:zt square is 0} \\
\partial_3(z_t^{n+1})&= \mu_{t}q_{21}^{nt}q_{22}^n n \, y_t z_t^n, & n,t\in\N_0,& \,  \epsilon^tq_{22}=1. \label{eq:zt n partial 3}
\end{align}
\end{lemma}

\begin{lemma}\label{lemma:x1, x12 commute with zt}
Let $\toba$ be a quotient algebra of $T(V)$. Assume that $x_1x_3=q_{12}x_3x_1$, and either
\begin{enumerate}
\item[(a)] \eqref{eq:rels B(V(1,2))}, or else
\item[(b)] \eqref{eq:rels-B(V(-1,2))-2}, $x_{21}x_3 =q_{12}^2 x_3x_{21}$
\end{enumerate}
hold in $\toba$. Then for all $n\in\N_0$,
\begin{align}\label{eq:x1, x12 commute with zt}
x_1z_n &= \epsilon^n q_{12}z_nx_1 \\
x_{21}z_n &= q_{12}^2z_nx_{21}.
\end{align}
\end{lemma}

\begin{lemma}\label{lemma:zt zk}
Let $\toba$ be a quotient algebra of $T(V)$, $\epsilon^2=q_{22}^2=1$.

\medbreak
\noindent \emph{\vi} Assume that \eqref{eq:zt zt+1 qcommute} and \eqref{eq:zt square is 0} hold in $\toba$.
Then for $0\le t<k\le |\sa|$,
\begin{align}\label{eq:bracket ztzk}
z_tz_k&=\epsilon^{tk} q_{21}^{k-t}q_{22} z_kz_t.
\end{align}

\medbreak
\noindent \emph{\vii} Assume that $z_t^2=0$ in $\toba$ for $t\in\N_0$ such that $\epsilon^t q_{22}=-1$. Then
$z_tz_{t+1}=q_{21}q_{22} z_{t+1}z_t$ in $\toba$.
\end{lemma}

In other words, \vii says that \eqref{eq:zt square is 0} for a specific $t$ implies \eqref{eq:zt zt+1 qcommute} for $t$.

\subsubsection{The Nichols algebra $\toba(\lstr( 1, \ghost))$}\label{subsubsection:lstr-11disc}

\begin{prop} \label{pr:lstr-11disc} Let $\ghost \in \I_{p-1}$. The algebra
$\toba(\lstr( 1, \ghost))$ is presented by generators $x_1,x_2, x_3$ and relations \eqref{eq:rels B(V(1,2))},
\eqref{eq:rels B(V(1,2))-x1p}, \eqref{eq:rels B(V(1,2))-x2p}, together with
\begin{align}
x_1x_3&=q_{12} \, x_3x_1,  \label{eq:lstr-rels&11disc-1} \\
z_{1+\ghost}&=0,  \label{eq:lstr-rels&11disc-qserre} \\
z_tz_{t+1}&=q_{12}^{-1} \, z_{t+1}z_t, &  0\le & t < \ghost, \label{eq:lstr-rels&11disc-2}
\\
z_t^p&=0, & 0\le & t \le  \ghost, \label{eq:lstr-rels&11disc-pot-p}\end{align}
The dimension of $\toba(\lstr( 1, \ghost))$ is $p^{\ghost + 3}$, since it has a PBW-basis
\begin{align*}
B=\{ x_1^{m_1} x_2^{m_2} z_{\ghost}^{n_{\ghost}} \dots z_1^{n_1} z_0^{n_0}: m_i, n_j \in \I_{0, p}\}.
\end{align*}.
\qed \end{prop}

\subsubsection{The Nichols algebra $\toba(\lstr( -1, \ghost))$}\label{subsubsection:lstr-1-1disc}

\begin{prop} \label{pr:lstr1-1disc} Let $\ghost \in \I_{p-1}$. The algebra
$\toba(\lstr( -1, \ghost))$ is presented by generators $x_1,x_2, x_3$ and relations  \eqref{eq:rels B(V(1,2))},
\eqref{eq:rels B(V(1,2))-x1p}, \eqref{eq:rels B(V(1,2))-x2p}, 
\eqref{eq:lstr-rels&11disc-1}, \eqref{eq:lstr-rels&11disc-qserre} and
\begin{align}\label{eq:lstr-rels&1-1disc}
z_t^2&=0, & 0\le& t\le \ghost.
\end{align}
The dimension of $\toba(\lstr( 1, \ghost))$ is $2^{\ghost + 1} p^{2}$, since it has a PBW-basis
\begin{align*}
B=\{ x_1^{m_1} x_2^{m_2} z_{\ghost}^{n_{\ghost}} \dots z_1^{n_1} z_0^{n_0}: n_i \in\{0,1\}, m_j \in \I_{0, p-1} \}.
\end{align*}
\qed
\end{prop}

\subsubsection{The Nichols algebra $\toba(\lstr_{-}( 1, \ghost))$}\label{subsubsection:lstr--11disc}

\begin{prop} \label{pr:lstr--11disc} Let $\ghost \in \I_{2p-1}\cap 2\Z$. The algebra
$\toba(\lstr_{-}( 1, \ghost))$ is presented by generators $x_1,x_2, x_3$ and relations \eqref{eq:rels-B(V(-1,2))-1},
\eqref{eq:rels-B(V(-1,2))-2}, \eqref{eq:rels-B(V(-1,2))-4},  \eqref{eq:rels-B(V(-1,2))-3}, 
\eqref{eq:lstr-rels&11disc-1} and
\begin{align}
z_{1+\ghost}&=0, \label{eq:lstr-rels&-11disc-1} \\
x_{21}z_0& = q_{12}^2 \, z_0x_{21},  \label{eq:lstr-rels&-11disc-2} \\
z_{2k+1}^2&=0, &  0\le & k < \ghost/2, \label{eq:lstr-rels&-11disc-3} \\
z_{2k} z_{2k+1}&= q_{12}^{-1} \, z_{2k+1}z_{2k}, & 0\le & k < \ghost/2. \label{eq:lstr-rels&-11disc-4}
\end{align}
The dimension of $\toba(\lstr( 1, \ghost))$ is $2^{\frac{\ghost}{2} +2} p^{\frac{\ghost}{2}+3}$, since it has a PBW-basis
\begin{align*}
B=\{ x_1^{m_1} x_{21}^{m_2} x_2^{m_3} z_{\ghost}^{n_{\ghost}} \dots z_1^{n_1} z_0^{n_0}: m_1, n_{2k+1} \in\{0,1\},\\ m_3 \in \I_{0,2p-1}, m_2, n_{2k} \in\I_{0,p-1} \}.
\end{align*}
\qed
\end{prop}

\subsubsection{The Nichols algebra $\toba(\lstr_{-}( -1, \ghost))$}\label{subsubsection:lstr--1-1disc}

\begin{prop} \label{pr:lstr-1-1disc} Let $\ghost \in \I_{2p-1}\cap 2\Z$. The algebra
$\toba(\lstr_{-}( -1, \ghost))$ is presented by generators $x_1,x_2, x_3$ and relations \eqref{eq:rels-B(V(-1,2))-1},
\eqref{eq:rels-B(V(-1,2))-2}, \eqref{eq:rels-B(V(-1,2))-4},  \eqref{eq:rels-B(V(-1,2))-3}, \eqref{eq:lstr-rels&11disc-1}, \eqref{eq:lstr-rels&-11disc-1}, \eqref{eq:lstr-rels&-11disc-2}
and
\begin{align}
z_{2k}^2&=0, &  0\le & k \le \ghost/2, \label{eq:lstr-rels&-1-1disc-1} \\
z_{2k-1} z_{2k}&= -q_{12}^{-1} z_{2k}z_{2k-1}, & 0< & k \le \ghost/2. \label{eq:lstr-rels&-1-1disc-2}
\end{align}
The dimension of $\toba(\lstr( 1, \ghost))$ is $2^{\frac{\ghost}{2} +3} p^{\frac{\ghost}{2}+2}$, since it has a PBW-basis
\begin{align*}
B=\{ x_1^{m_1} x_{21}^{m_2} x_2^{m_3} z_{\ghost}^{n_{\ghost}} \dots z_1^{n_1} z_0^{n_0}: m_1, n_{2k} \in\{0,1\}, \\ m_3 \in\I_{0,2p-1}, m_2, n_{2k-1} \in\N_0 \}.
\end{align*}
\qed
\end{prop}

\subsubsection{The Nichols algebra $\toba(\lstr( \omega, 1))$}\label{subsubsection:lstr-1omega1}

\begin{prop} \label{pr:lstr1omega1} Let $\omega \in \G'_3$. The algebra
$\toba(\lstr( \omega, 1))$ is presented by generators $x_1,x_2, x_3$ and relations  \eqref{eq:rels B(V(1,2))},
\eqref{eq:rels B(V(1,2))-x1p}, \eqref{eq:rels B(V(1,2))-x2p},\eqref{eq:lstr-rels&11disc-1}, 
\begin{align}\label{eq:lstr1omega1-qserre}
z_2 & =0, 
\\ \label{eq:lstr1omega1-z0cube}
z_0^3 & =0,
\\ \label{eq:lstr-rels&1omega1}
z_1^3&=0,
\\ \label{eq:lstr-rels&1omega10} z_{1,0}^3&=0.
\end{align}

The dimension of $\toba(\lstr( \omega, 1))$ is $3^3p^{2}$, since it has a PBW-basis
\begin{align*}
B=\{ x_1^{m_1} x_2^{m_2} z_1^{n_1} z_{1,0}^{n_2} z_0^{n_3}: m_i\in\I_{0,p-1},   n_j \in\I_{0,2} \}.
\end{align*} \qed
\end{prop}

\subsection{Mild interaction}\label{subsection:mild}
\emph{We assume in this Subsection that $q_{12}q_{21}=-1 = \epsilon$, $a=1$, $q_{22}= -1$}.
The corresponding braided vector space is denoted $\cyc_1$, as above. We proceed as above but now
the elements  $z_n = (\ad_c x_2)^n x_3$ are not enough to describe $K^1$ and we need
$ f_n=\ad_c x_1 (z_n)$, $n = 0,1$.
Then
\begin{align}\label{rem:nichols-mild-relations}
\begin{aligned}
x_1z_0 & =f_0+q_{12}z_0x_1, \\ x_1z_1 & =f_1-q_{12}z_1x_1+q_{12}f_0x_1, \\ x_2z_0 & =z_1+q_{12}z_0x_2.
\end{aligned}
\end{align}

\begin{prop} \label{prop:pm1bp-mild} The Nichols algebra
$\toba(\cyc_1)$ is presented by generators $x_1,x_2, x_3$ and relations \eqref{eq:rels-B(V(-1,2))-1},
\eqref{eq:rels-B(V(-1,2))-2}, \eqref{eq:rels-B(V(-1,2))-4},  \eqref{eq:rels-B(V(-1,2))-3},
\begin{align}\label{eq:nichols-mild-relation2}
\begin{aligned}
x_2z_1+q_{12}z_1x_2&=\frac 1 2 f_1+q_{12}f_0x_2, \\ x_2f_1 &=q_{12}f_1x_2,  \quad x_2f_0+q_{12}f_0x_2=-f_1,  \\
z_0^2 &=0, \quad f_0^2 = 0, \quad z_1^2 =0, \qquad f_1^2 = 0.
\end{aligned}
\end{align}
The dimension of $\toba(\cyc_1)$ is $ 64 p^{2}$, since it has a PBW-basis
\begin{align*}
B=\{ x_1^{m_1} x_{21}^{m_2} x_2^{m_3} f_1^{n_1} f_0^{n_2} z_1^{n_3} z_0^{n_4}: m_1,n_i \in\{0,1\}, m_2,m_3 \in\I_{p}\}.
\end{align*}
\qed
\end{prop}

\subsection{Realizations}\label{subsec:realizations-block-pt}
Let $H$ be a Hopf algebra, $(g_1, \chi_1, \eta)$ a YD-triple and $(g_2, \chi_2)$
a YD-pair for $H$, see \S \ref{subsec:realizations-block}. 
Let $(V, c)$ be a braided vector space  with braiding 
\eqref{eq:braiding-block-point}.
Then  $\cV_{g_1}(\chi_1, \eta) \oplus  \ku_{g_2}^{\chi_2} \in \ydh$
is a \emph{principal realization} of $(V, c)$  over  $H$ if
\begin{align*}
q_{ij}&= \chi_j(g_i),& &i, j\in \I_{2};& a&= q_{21}^{-1}\eta(g_2).
\end{align*}
Thus $(V, c) \simeq \cV_{g_1}(\chi_1, \eta) \oplus  \ku_{g_2}^{\chi_2}$ 
as braided vector space.
Hence, if $H$ is finite-dimensional and $(V, c)$ is as in Table \ref{tab:toba-finitedim-block-point},
then $\toba \big(\cV_{g_1}(\chi_1, \eta) \oplus  \ku_{g_2}^{\chi_2}\big) \# H$ is a finite-dimensional Hopf algebra.
Examples of finite-dimensional pointed Hopf algebras 
$A = \NA\big(\cV_{g_1}(\chi_1, \eta) \oplus  \ku_{g_2}^{\chi_2}\big) \# \ku \Gamma$ 
like this are listed in Table \ref{tab:hopf-finitedim-block-point}.
In all cases $\Gamma$ is a product of two cyclic groups, $g_1=(1,0)$, $g_2=(0,1)$ and $\chi_j(g_i)=1$ if $i\neq j$; hence it remains to fix the value of $q_{12}$.

\begin{table}[ht]
\caption{{\small Pointed Hopf algebras $K$ from a block and a point}}\label{tab:hopf-finitedim-block-point}
\begin{center}
\begin{tabular}{|c|c|c|c|c|}
\hline $V$ & {\scriptsize diagram}    & $\Gamma$  &  $q_{12}$
& {\small $\dim A$}  \\
\hline
$\lstr( 1, \ghost)$ & $\xymatrix{\boxplus \ar  @{-}[r]^{\ghost}  & \overset{1}{\bullet}}$  
& $\Z/p \times \Z/p$ & $1$ &    $p^{\sa + 5}$
\\\hline
$\lstr( -1, \ghost)$ & $\xymatrix{\boxplus \ar  @{-}[r]^{\ghost}  & \overset{-1}{\bullet}}$& $\Z/p \times \Z/2p$ & $1$  &   $2^{\sa + 2}p^{4}$
\\ \hline
$\lstr(\omega, 1)$& $\xymatrix{\boxplus \ar  @{-}[r]^{1}  & \overset{\omega}{\bullet}}$&
$\Z/p \times \Z/3p$ & $1$ &   $3^4 p^{4}$
\\ \hline
$\lstr_{-}(1, \ghost)$
& $\xymatrix{\boxminus \ar  @{-}[r]^{\ghost}  & \overset{1}{\bullet}}$
& $\Z/2p \times \Z/p$ & $1$
& $2^{\frac{\sa}{2} + 3}p^{\frac{\sa}{2} + 5}$
\\ \hline
$\lstr_{-}(-1, \ghost)$
& $\xymatrix{\boxminus \ar  @{-}[r]^{\ghost}  & \overset{-1}{\bullet}}$
& $\Z/2p \times \Z/2p$ & $\pm 1$
&  $2^{\frac{\sa}{2} + 5}p^{\frac{\sa}{2} + 4}$
\\ \hline
$\cyc_1$&$\xymatrix{\boxminus \ar  @{-}[r]^{(-1, 1)}  &\overset{-1}{\bullet} }$ 
& $\Z/2p \times \Z/2p$ & $\pm 1$
&$256p^4$
\\\hline
\end{tabular}
\end{center}
\end{table}

\section{One block and several points}\label{sec:yd-dim>3}
\subsection{The setting and the main result}\label{subsection:YD>3-setting}

Let $\theta \in \N_{\ge 3}$,  $\I_{2,\theta} = \I_{\theta} - \{1\}$,
$\Iw_\theta = \I_{\theta} \cup \{\fudos\}$. Let $\lfloor i\rfloor$ be the largest integer $\leq i$.
We start from the data
\begin{align*}
(q_{ij})_{i,j \in \I_{\theta}} &\in (\kut)^{\theta \times \theta}, \quad q_{11}^2 = 1;&
(a_2, \dots, a_\theta) \in \ku^{\I_{2,\theta}}.
\end{align*}
We assume that $q_{11}= 1 =: a_1$.
Let $(V, c)$ be the braided vector space of dimension $\theta + 1$,
with a basis $(x_i)_{i\in\Iw_{\theta}}$ and braiding given  by 
\begin{align}\label{eq:braiding-block-several-point}
c(x_i \otimes x_j) &= \begin{cases}
q_{\lfloor i\rfloor j} x_j  \otimes x_i, &i\in \Iw_{\theta},\, j\in \I_{\theta};\\
q_{\lfloor i\rfloor 1} (x_{\fudos} + a_{\lfloor i\rfloor} x_1) \otimes x_{i}, & i\in \Iw_{\theta},\, j =\fudos.
\end{cases}
\end{align}

We say that the block and the points have \emph{discrete ghost} if $a_j \in \Fp^{\I_{2,\theta}}$, $(a_j)\neq 0$.
When this is the case, we pick the representative $\sa_j \in \Z$ of $2a_j$ by imposing
$\sa_j \in \{1- p, \dots, -1,0 \}$, 
and set $\ghost_j =-\sa_j$.
The \emph{ghost} between the block and the points is the  vector $\ghost = (\ghost_{j})_{j \in \I_{2,\theta}}$ given by
\begin{align}
\ghost = -(\sa_{j})_{j \in \I_{2,\theta}} \in \N_0^{\I_{2,\theta}}.
\end{align}

\medbreak

The braided subspace $V_1$ spanned by $x_1, x_{\fudos}$ is $\simeq \cV(1, 2)$, while
$V_{\diag}$ spanned by $(x_i)_{i\in\I_{2,\theta}}$ is of diagonal type. Obviously,
\begin{align}\label{eq:v=v1+v2}
V &= V_1 \oplus V_{\diag}.
\end{align}
Let $\X$ be the set of connected components of the Dynkin diagram of the matrix $\bq = (q_{ij})_{i, j \in \I_{2, \theta}}$.
If $J\in \X$, then we set  $J' = \I_{2, \theta} - J$,
\begin{align*}
V_J &= \sum_{j \in J} \ku_{g_j}^{\chi_j},&
\ghost_J &= (\ghost_{j})_{j \in J}.
\end{align*}

We shall use the results and notations from the preceding Sections, but with $\fudos$ replacing 2 when appropriate,
e. g. $x_{\fudos 1} = x_{\fudos}x_1 - x_1 x_{\fudos}$.
Let
\begin{align*}
K &=\NA (V)^{\mathrm{co}\,\NA (V_1)} \text{ and }&  K^1 &= \ad_c\NA (V_1) (V_{\diag})
\in {}^{\NA (V_1)\# \ku \Gamma}_{\NA (V_1)\# \ku \Gamma}\mathcal{YD},
\\
&\text{so that}& \NA(V) &\simeq K \# \NA (V_1), \quad K \simeq \NA(K^1) 
\end{align*}
Let
\begin{align}\label{eq:zjn}
z_{j,n} &:= (ad_c x_{\fudos})^n x_j,& j&\in\I_{2, \theta},& n&\in\N_0.
\end{align}

For all $i,j\in \I_{2,\theta}$, $n\in \N_0$, we have as in \cite[\S 5.2.1]{aah-triang} that
\begin{align}
\label{eq:1block+points-action}
g_1\cdot z_{j,n} &= q_{1j}z_{j,n},& &\text{(by Lemma \ref{le:-1bpz})}
\\\label{eq:1block+points-bis}
g_i\cdot z_{j,n} &= q_{i1}^nq_{ij}z_{j,n}, &&
\end{align}

\begin{lemma}\label{lemma:braiding-K-weak-block-points} 
The  braided vector space $K^1$ is of diagonal type in the basis
\begin{align}\label{eq:block-points-baseK1}
(z_{j,n})_{j\in \I_{2,\theta}, 0\le n\le  \ghost_j}
\end{align}
with braiding matrix
\begin{align*}
(\bp_{im, jn})_{\substack{i,j \in \I_{2,\theta}, \\ 0\le m\le  \ghost_i, \,
0\le n \le \ghost_j }} &= (q_{i1}^nq_{1j}^m  q_{ij})_{\substack{ i,j \in \I_{2,\theta}, \\ 0\le m\le  \ghost_i, \, 
0\le n \le \ghost_j }}.
\end{align*}
Hence, the corresponding generalized  Dynkin diagram  has labels
\begin{align*}
\bp_{im,im}&=  q_{ii},& \bp_{im, jn}\bp_{jn, im}&= q_{ij}q_{ji}, & &(i,m)\neq (j,n).
\end{align*}
\end{lemma}

\pf The proof in \cite[Lemma 7.2.5]{aah-triang} applies as the combinatorial numbers appearing there are not zero.
\epf

Let $K_J$ be the braided vector subspace of $K^1$ spanned by $(z_{j,n})_{\substack{j\in J,\\ 0\le n\le  \ghost_j}}$.

\begin{coro}\label{cor:conncomp}
The braided subspaces corresponding to the connected components of the Dynkin diagram of $K^1$ are  $K_J$, $J \in \X$. Hence
\begin{align}\label{eq:points-block-gkd-K}
\dim K = \dim \toba(K^1) &=  \prod_{J\in \X} \dim \toba(K_J). \qed
\end{align}
\end{coro}
Observe that if $\ghost_J = 0$, then $K_J = V_J$.

\smallbreak
In this Section, we shall prove part \ref{item:block-points} of Theorem \ref{thm:main-intro}.

\begin{theorem}\label{thm:points-block-eps1} Let $V$ be a braided vector space with braiding \eqref{eq:braiding-block-several-point}. 
Assume that for every $J \in \X$, either  $\ghost_J = 0$, or else $\dim V_J = 1$ 
and $V_1 \oplus V_j$ is as in Table \ref{tab:toba-finitedim-block-point}, or else
$V_J$ is  as in Table \ref{tab:toba-finitedim-block-points}. Then 
\begin{align}\label{eq:points-block-gkd}
\dim \toba(V) &= p^2 \prod_{J\in \X} \dim \toba(K_J) < \infty.
\end{align}
\end{theorem}

\pf 
By Corollary \ref{cor:conncomp} we reduce to connected components in $\X$.
If $J\in \X$ has weak interaction and $\ghost_J = 0$, then
$\NA(V) \simeq \NA(V_1 \oplus V_{J'})  \, \underline{\otimes} \, \NA(V_J)$,
hence $\dim \NA(V) = \dim\NA(V_1 \oplus V_{J'})  \dim \NA(V_J)$.

If $J\in \X$ is a point, 
then Theorem \ref{thm:point-block} applies. 
We need to analyze those $J$ with  $\vert J \vert  \geq 2$ and $\ghost_J \neq 0$.

Below we denote $\imath\, = \sqrt{-1}$.

\medbreak \noindent $\lstr(A_{2})$, $\ghost_J = (1,0)$: Here $K_J$ is of Cartan type   $A_3$ 
and $\dim \toba(K_J) = 2^{6}$.

\medbreak \noindent $\lstr(A_{j})$, $j> 2$, $\ghost_J = (1,0)$: Here $K_J$ is of Cartan type   $D_{j-1}$ 
and $\dim \toba(K_J) = 2^{j(j+1)}$.

\medbreak \noindent $\lstr(A_{2}, 2)$: Here $K_J$ is of Cartan type   $D_4$ 
and $\dim \toba(K_J) = 2^{12}$.

\medbreak \noindent $\lstr(A(1\vert 0)_2; \omega)$, provided that $p >3$: Here
$K_J$ is of diagonal type in the basis $1 = z_{i,0}$, $2 = z_{i,1}$, $3 = z_{j,0}$ with diagram 
\begin{align*}
\xymatrix{\overset{-1}{\underset{2}{\circ}} \ar  @{-}[rd]^{\omega} &
\\ \overset{-1}{\underset{1}{\circ}} \ar  @{-}[r]^{\omega}  &
\overset{-1}{\underset{3}{\circ}}}
\end{align*}
Then $\dim \toba (K_J) < \infty$ by  \cite[Table 2, row 15]{H-classif};
it has the same root system as $\mathfrak g(2,3)$ and dimension $2^73^4$, see \cite[8.3.4]{AA17}.

\medbreak \noindent $\lstr(A(1\vert 0)_1; \omega)$ and $\lstr(A(1\vert 0)_3; \omega)$, provided that $p >3$: 
Here $K_J$ is of diagonal type in the basis  $1 = z_{i,0}$, $2 = z_{j,1}$, $3 = z_{j,0}$, respectively 
$1 = z_{i,0}$, $2 = z_{i,1}$, $3 = z_{j,0}$, and its diagram is
\begin{align*}
& \xymatrix{ &  \overset{-1}{\underset{2}{\circ}} \ar  @{-}[ld]_{\omega^{2}}
\\
 \overset{\omega}{\underset{1}{\circ}} \ar  @{-}[r]^{\omega^{2}}  & \overset{-1}{\underset{3}{\circ}}
}
&&\xymatrix{\overset{\omega}{\underset{2}{\circ}}\ar  @{-}[d]_{\omega^{2}} \ar  @{-}[rd]^{\omega^2} &
\\ \overset{\omega}{\underset{1}{\circ}} \ar  @{-}[r]^{\omega^2}  &
\overset{-1}{\underset{3}{\circ}}}
\end{align*}
Then $\dim \toba (K_J) < \infty$ by  \cite[Table 2, row 8]{H-classif}, respectively \cite[Table 2, row 15]{H-classif}.
In the first case it has the same root system as $\mathfrak{sl}(2\vert 2)$ and dimension $2^43^2$, see \cite[5.1.8]{AA17}.
In the second case it has the same root system as $\mathfrak g(2,3)$ and dimension $2^73^4$, see \cite[8.3.4]{AA17}.

\medbreak \noindent R$\lstr(A(1\vert 0)_1; r)$: analogous to $\lstr(A(1\vert 0)_1; \omega)$; same root system as $\mathfrak{sl}(2\vert 2)$ and dimension $2^4N^2$, see \cite[5.1.8]{AA17}.

\medbreak \noindent $\lstr(A(2 \vert 0)_1; \omega)$ and $\lstr(D(2 \vert 1); \omega)$, $p >3$: In both cases, 
$\dim \toba (K_J) = 2^83^9$; it has the same root system as $\mathfrak g(3,3)$, see \cite[Table 3, row 18]{H-classif}, \cite[8.4.5]{AA17}.
\epf

\subsection{The presentation of the Nichols algebras}\label{subsection:points-block-presentation}
We give defining relations and an explicit PBW basis of $\toba(V)$, for all $V$ as in
Theorem \ref{thm:points-block-eps1},   
assuming that the Dynkin diagram of $V_{\diag}$ is connected, i.e. $V_{\diag} = V_J$, where $J = \I_{2, \theta}$.
Essentially the relations are the same as in \cite{aah-triang} up to adding the suitable $p$-powers;
we omit the  proofs as they are minor variations of those in \emph{loc. cit.}
The passage from connected $V_{\diag}$ to the general case is standard, just  add  the quantum commutators between points in different components.
Since the case $\vert J\vert =  1$ was treated in \S \ref{sec:yd-dim3}, we also suppose  that $\vert J\vert > 1$.
These braided vector spaces have names given in \cite{aah-triang}, see Table \ref{tab:toba-finitedim-block-points}.
The braided vector subspace $V_1\oplus \ku x_2$ of such $V$ is of type
\begin{itemize}
  \item $\lstr(-1, 2)$ when $V$ is of type $\lstr(A_{2}, 2)$,
  \item $\lstr(\omega, 1)$ when $V$ is of type $\lstr(A(1\vert 0)_3; \omega)$, or
  \item $\lstr(-1, 1)$ for all the other cases.
\end{itemize}
Thus the subalgebra generated by $V_1\oplus \ku x_2$ is a Nichols algebra.
We recall its relations up to the change of index with respect to \S  \ref{sec:yd-dim3}; the $2$ and $3$ there are now $\fudos$ and $2$.
As in \eqref{eq:xijk}, we set   $x_{i_1i_2 \dots i_M} = \ad_c x_{i_1}\, x_{i_2 \dots i_M}$.
Also, we have now $z_n = (\ad_c x_{\fudos})^n x_2$, $n\in\N_0$.

\smallbreak
First, the defining relations of $\toba(\lstr(-1, 1))$ are  
\begin{align}
\label{eq:rels B(V(1,2))-bis}  &x_{\fudos}x_1-x_1x_{\fudos}+\frac{1}{2}x_1^2,
\\
\tag{\ref{eq:rels B(V(1,2))-x1p}}
& x_{1}^p, 
\\\label{eq:rels B(V(1,2))-x2p-bis}
& x_{\fudos}^p, 
\\
\label{eq:lstr-rels&11disc-1-bis} & x_1x_2-q_{12} \, x_2x_1,
\\
\label{eq:lstr-rels&11disc-qserre-g=1} & (\ad_c x_{\fudos})^2 x_2,
\\
\label{eq:lstr-rels&1-1disc-g=1} & x_2^2, \, x_{\fudos 2}^2.
\end{align}
Second, the defining relations of $\toba(\lstr(-1, 2))$ are \eqref{eq:rels B(V(1,2))-bis},
\eqref{eq:rels B(V(1,2))-x1p}, \eqref{eq:rels B(V(1,2))-x2p-bis}, \eqref{eq:lstr-rels&11disc-1-bis}, \eqref{eq:lstr-rels&1-1disc-g=1} and 
\begin{align}
\label{eq:lstr-rels&11disc-qserre-g=2} & (\ad_c x_{\fudos})^3 x_2,
\\
\label{eq:lstr-rels&1-1disc-g=2} & x_{\fudos\fudos 2}^2.
\end{align}
Third, the defining relations of $\toba(\lstr(\omega, 1))$ are \eqref{eq:rels B(V(1,2))-bis},
\eqref{eq:rels B(V(1,2))-x1p}, \eqref{eq:rels B(V(1,2))-x2p-bis}, 
\eqref{eq:lstr-rels&11disc-1-bis}, \eqref{eq:lstr-rels&11disc-qserre-g=1} and
\begin{align}
\label{eq:lstr1omega1-z0cube-bis} &  x_2^3, 
\\ \label{eq:lstr-rels&1omega1-bis}
& x_{\fudos 2}^3,
\\
\label{eq:lstr-rels&1omega13-bis}  &[x_{\fudos 2}, x_2]_c^3.
\end{align}

We also observe that, since $q_{1j}q_{j1}=1$ and $\ghost_j=0$, we have
\begin{align}\label{eq:lstr-rels-ghost-int-trivial}
x_1x_j &= q_{1j} x_jx_1, & x_{\fudos} x_j &= q_{1j} x_j x_{\fudos},& j &\in \I_{3, \theta}.
\end{align}

\subsubsection{The Nichols algebra $\toba(\lstr(A(1\vert 0)_1; r))$, $r\in\G_N'$, $N\geq 3$} \label{subsubsection:lstr-a(10)1}
Let
\begin{align}\label{eq:lstr-def-ztt-1}
\ztt_{12}&=x_{\fudos 23}, & \ztt_{123}&=[x_{\fudos 2}, x_{23}]_c.
\end{align}

\begin{prop} \label{pr:lstr-a(10)1}
The algebra $\toba(\lstr(A(1\vert 0)_1; r))$ is presented by generators $x_i$, $i\in \Iw_3$, and relations
\eqref{eq:rels B(V(1,2))-bis}, \eqref{eq:rels B(V(1,2))-x1p}, \eqref{eq:rels B(V(1,2))-x2p-bis},
 \eqref{eq:lstr-rels&11disc-1-bis}, \eqref{eq:lstr-rels&11disc-qserre-g=1}, \eqref{eq:lstr-rels&1-1disc-g=1}
\eqref{eq:lstr-rels-ghost-int-trivial}, and
\begin{align}
\label{eq:lstr-rels-a10-diagpart}
(\ad_c x_3)^2 x_2 &=0, & x_3^N &= 0,
\\
\label{eq:lstr-rels-a10-1}
\ztt_{123}^N&=0,
\end{align}
The set
\begin{multline*}
B =\big\{ x_1^{m_1} x_{\fudos}^{m_2} x_{\fudos 2}^{n_{1}}  \ztt_{12}^{n_{2}}\ztt_{123}^{n_{3}}x_{3}^{n_{4}}x_{23}^{n_{5}}x_{2}^{n_{6}} : \\
0 \le n_{1}, n_{2}, n_{5}, n_{6} < 2, \,  0\le n_{3}, n_{4} < N, \, 0 \le m_1, m_2 < p \big\}
\end{multline*}
is a basis of $\toba(\lstr(A(1\vert 0)_1; r))$ and $\dim \toba(\lstr(A(1\vert 0)_1; r)) = p^2 2^4 N^2$. 
\qed
\end{prop}

\subsubsection{The Nichols algebra $\toba(\lstr(A(1\vert 0)_2; \omega))$} \label{subsubsection:lstr-a(10)2}

\begin{prop} \label{pr:lstr-a(10)2}
The algebra $\toba(\lstr(A(1\vert 0)_2; \omega))$ is presented by generators $x_i$, $i\in \Iw_3$, and relations
\eqref{eq:rels B(V(1,2))-bis}, \eqref{eq:rels B(V(1,2))-x1p}, \eqref{eq:rels B(V(1,2))-x2p-bis},
\eqref{eq:lstr-rels&11disc-1-bis}, \eqref{eq:lstr-rels&11disc-qserre-g=1}, \eqref{eq:lstr-rels&1-1disc-g=1},
\eqref{eq:lstr-rels-ghost-int-trivial}, and
\begin{align}\label{eq:lstr-rels-a10-2-diagpart}
x_2^2 &=0, & x_3^2 &= 0, & x_{23}^3 &= 0,
\\
\label{eq:lstr-rels-a10-2}
\ztt_{12}^3&=0, & \ztt_{123}^6&=0, & [\ztt_{123},x_3]_c^3&=0,
\end{align}
The set
\begin{multline*}
B =\big\{ x_1^{m_1} x_{\fudos}^{m_2} x_{\fudos 2}^{n_1} \ztt_{12}^{n_2} [\ztt_{12},[\ztt_{12},\ztt_{123}]_c]_c^{n_3}
[\ztt_{12},\ztt_{123}]_c^{n_4} \ztt_{123}^{n_5}  \\
[\ztt_{123},x_3]_c^{n_6} [\ztt_{123},x_{32}]_c^{n_7} x_3^{n_8} x_{32}^{n_9} x_2^{n_{10}}:  0 \le m_1, m_2,< p, \\
0\le n_2,n_6,n_9<3, \, 0\le n_5<6, \,0\le n_1,n_3,n_4,n_7,n_8,n_{10}<2 \big\}
\end{multline*}
is a basis of $\toba(\lstr(A(1\vert 0)_2; \omega))$ and $\dim \toba(\lstr(A(1\vert 0)_2; \omega)) = p^22^73^4$.
\qed
\end{prop}

\subsubsection{The Nichols algebra $\toba(\lstr(A(1\vert 0)_3; \omega))$} \label{subsubsection:lstr-a(10)3}

\begin{prop} \label{pr:lstr-a(10)3}
The algebra $\toba(\lstr(A(1\vert 0)_3; \omega))$ is presented by generators $x_i$, $i\in \Iw_3$, and relations
\eqref{eq:rels B(V(1,2))-bis}, \eqref{eq:rels B(V(1,2))-x1p}, \eqref{eq:rels B(V(1,2))-x2p-bis},
\eqref{eq:lstr-rels&11disc-1-bis},  \eqref{eq:lstr-rels&11disc-qserre-g=1}, \eqref{eq:lstr1omega1-z0cube-bis}, \eqref{eq:lstr-rels&1omega1-bis}, \eqref{eq:lstr-rels&1omega13-bis}
\eqref{eq:lstr-rels-ghost-int-trivial}, 
\begin{align}
\label{eq:lstr-rels-a10-3-diagpart}
x_2^3 &=0, & x_3^2 &= 0, & x_{223} &= 0,
\\ \label{eq:lstr-rels-a10-3}
x_{\fudos 2}x_{\fudos 23}&+q_{13}q_{23}x_{\fudos 23}x_{\fudos 2}=0, &  \ztt_{123}^6&=0.
\end{align}
The set
\begin{multline*}
B =\big\{ x_1^{m_1} x_{\fudos}^{m_2}x_{\fudos 2}^{n_1} \ztt_{12}^{n_2} [\ztt_{12},\ztt_{13}]_c^{n_3} \ztt_{123}^{n_4} [\ztt_{123},\ztt_{13}]_c^{n_5} [\ztt_{13},x_{32}]_c^{n_6}
\\ \ztt_{13}^{n_7} x_3^{n_8} x_{32}^{n_9} x_2^{n_{10}}:
0 \le m_1, m_2< p,  \, 0\le n_2,n_3,n_5,n_6,n_8,n_9<2,  \\ 0\le n_1,n_7,n_{10}<3, \, 0\le n_4<6 \big\}
\end{multline*}
is a basis of $\toba(\lstr(A(1\vert 0)_3; \omega))$ and $\dim \toba(\lstr(A(1\vert 0)_3; \omega)) = p^22^73^4$. \qed
\end{prop}

\subsubsection{The Nichols algebra $\toba(\lstr(A(2\vert 0)_1; \omega))$} \label{subsubsection:lstr-a(20)1}

\begin{prop} \label{pr:lstr-a(20)1}
The algebra $\toba(\lstr(A(2\vert 0)_1; \omega))$ is presented by generators $x_i$, $i\in \Iw_4$, and relations
\eqref{eq:rels B(V(1,2))-bis}, \eqref{eq:rels B(V(1,2))-x1p}, \eqref{eq:rels B(V(1,2))-x2p-bis},
\eqref{eq:lstr-rels&11disc-1-bis}, \eqref{eq:lstr-rels&11disc-qserre-g=1}, \eqref{eq:lstr-rels&1-1disc-g=1},
\eqref{eq:lstr-rels-ghost-int-trivial}, and
\begin{align}\label{eq:lstr-rels-a(20)1-diagpart}
&\begin{aligned}
x_{24}&=0, & x_{332}&=0, & x_{334}&=0, & x_{443}&=0, 
\end{aligned}
\\ \label{eq:lstr-rels-a(20)1-diagpart-2}
&\begin{aligned}
x_2^2 &=0, & x_3^3&=0, & x_{34}^3&=0, & x_4^3&=0,
\end{aligned}
\\ \label{eq:lstr-rels-a(20)1}
&\begin{aligned}
{[x_{\fudos 23},x_2]_c}^3&=0, & {[[x_{\fudos 23},x_2]_c, [x_{\fudos 23},x_{24}]_c]_c}^3&=0,  \\
{[x_{\fudos 23},x_{24}]_c}^3&=0, & {[ [x_{\fudos 23},x_2]_c, [[x_{\fudos 23},x_{24}]_c,x_2]_c]_c}^3&=0,\\
{[[x_{\fudos 23},x_{24}]_c,x_2]_c}^3&=0, & { [ [x_{\fudos 23},x_{24}]_c, [[x_{\fudos 23},x_{24}]_c,x_2]_c]_c}^3&=0.
\end{aligned}
\end{align}

The set
\begin{multline*}
B =\big\{ x_1^{m_1} x_{\fudos}^{m_2}
x_{\fudos 2}^{n_1} \ztt_{12}^{n_2} [\ztt_{12},\ztt_{1234}]_c^{n_3} \ztt_{123}^{n_4} [\ztt_{123},\ztt_{1234}]_c^{n_5}
[\ztt_{123},[\ztt_{1234}, x_3]_c]_c^{n_6} 
\\
\ztt_{1234}^{n_7} [\ztt_{1234}, [\ztt_{1234}, x_3]_c]_c^{n_8} [\ztt_{1234}, x_3]_c^{n_{9}} [\ztt_{1234}, x_{32}]_c^{n_{10}} x_{\fudos 234}^{n_{11}} x_3^{n_{12}} x_{32}^{n_{13}} 
\\
x_{324}^{n_{14}} x_{34}^{n_{15}} x_2^{n_{16}} x_4^{n_{17}}: \, 0\le n_{1},n_{2},n_{3},n_{10},n_{11},n_{13},n_{14},n_{15} <2,
\\
0 \le m_1, m_2 < p, \, 0\le n_{4},n_{5},n_{6},n_{7},n_{8},n_{9},n_{12},n_{16},n_{17} <3 \big\}
\end{multline*}
is a basis of $\toba(\lstr(A(2\vert 0)_1; \omega))$ and $\dim \toba(\lstr(A(2\vert 0)_1; \omega)) = p^2 2^83^9$.
\qed
\end{prop}

\subsubsection{The Nichols algebra $\toba(\lstr(D(2\vert 1); \omega))$} \label{subsubsection:lstr-D(21)}

\begin{prop} \label{pr:lstr-d(21)}
The algebra $\toba(\lstr(D(2\vert 1); \omega))$ is presented by generators $x_i$, $i\in \Iw_4$, and relations
\eqref{eq:rels B(V(1,2))-bis}, \eqref{eq:rels B(V(1,2))-x1p}, \eqref{eq:rels B(V(1,2))-x2p-bis},
\eqref{eq:lstr-rels&11disc-1-bis}, \eqref{eq:lstr-rels&11disc-qserre-g=1}, \eqref{eq:lstr-rels&1-1disc-g=1},
\eqref{eq:lstr-rels-ghost-int-trivial}, and
\begin{align}
\label{eq:lstr-rels-d(21)}
&\begin{aligned}
&[[[x_{\fudos 23},x_{24}]_c, x_3]_c, x_3]_c^3=0, &  
&[[x_{\fudos 23},x_{24}]_c,x_3]_c^3=0, &&
\\
& [[x_{\fudos 23},x_{24}]_c,x_{334}]_c^3=0, & &[x_{\fudos 23},x_{24}]_c^3=0, & &[x_{\fudos 23},x_2]_c^3=0, 
 \end{aligned}
\\
\label{eq:lstr-rels-d(21)-diagpart-1}
&\begin{aligned}
x_{24}&=0, & x_{332}&=0, & x_{443}&=0, & [[x_{234}&,x_3]_c,x_3]_c=0,
\end{aligned}
\\ \label{eq:lstr-rels-d(21)-diagpart-2}
&\begin{aligned}
x_2^2&=0, & x_3^3&=0, & x_{34}^3&=0, & x_{334}^3&=0, \quad x_4^3=0.
\end{aligned}
\end{align}
The set
\begin{multline*}
B =\{ x_1^{m_1} x_{\fudos}^{m_2} x_{\fudos 2}^{n_1} \ztt_{12}^{n_2} \ztt_{123}^{n_3} \ztt_{1234}^{n_4} [\ztt_{1234},x_3]_c^{n_5}
[[\ztt_{1234}, x_3]_c, x_3]_c^{n_6} [\ztt_{1234},x_{334}]_c^{n_7}
\\
x_{\fudos 234}^{n_8} [x_{\fudos 234}, x_3]_c^{n_{9}} x_{3}^{n_{10}} x_{32}^{n_{11}} x_{324}^{n_{12}} [x_{324},x_3]_c^{n_{13}} x_{334}^{n_{14}} x_{34}^{n_{15}} x_2^{n_{16}} x_4^{n_{17}}:
\\
0\le m_1,m_2 < p, \, 0\le n_{1},n_{2},n_{8},n_{9},n_{11},n_{12},n_{13},n_{16} <2,
\\
0\le n_{3},n_{4},n_{5},n_{6},n_{7},n_{10},n_{14},n_{15},n_{17} <3 \}.
\end{multline*}
is a basis of $\toba(\lstr(D(2\vert 1); \omega))$ and $\dim \toba(\lstr(D(2\vert 1); \omega)) = p^22^83^9$. \qed
\end{prop}

\subsubsection{The Nichols algebra $\toba(\lstr(A_2, 2))$} \label{subsubsection:lstr-a-22}

\begin{prop} \label{pr:lstr-a-22}
The algebra $\toba(\lstr(A_2,2))$ is presented by generators $x_i$, $i\in \Iw_3$, and relations
\eqref{eq:rels B(V(1,2))-bis}, \eqref{eq:rels B(V(1,2))-x1p}, \eqref{eq:rels B(V(1,2))-x2p-bis}, \eqref{eq:lstr-rels&11disc-1-bis}, \eqref{eq:lstr-rels&1-1disc-g=1}, \eqref{eq:lstr-rels&11disc-qserre-g=2}, \eqref{eq:lstr-rels&1-1disc-g=2},
\eqref{eq:lstr-rels-ghost-int-trivial}, and

\begin{align}\label{eq:lstr-rels-a-22-diagpart}
x_2^2 &=0, \qquad x_3^2 =0, \qquad x_{32}^2=0.
\\\label{eq:lstr-rels-a-22}
&\begin{aligned}
& [x_{\fudos\fudos 23}, x_2]_c, x_{\fudos 2}]_c^2=0,& &x_{\fudos\fudos 23}^2=0, &
& [x_{\fudos\fudos 23}, x_{\fudos 2}]_c^2=0, 
 \\ 
& [x_{\fudos\fudos 23},x_2]_c^2=0, &
& x_{3\fudos 2}^2=0, &
& [x_{32},x_{\fudos 2}]_c^2=0, & \\
& [[x_{\fudos\fudos 23}, x_2]_c, x_{\fudos 2}]_c, x_3]_c^2=0.
\end{aligned}
\end{align}
The set $B=$
\begin{align*}
\{ x_1^{m_1} x_{\fudos}^{m_2} x_{\fudos\fudos 2}^{n_1} x_{\fudos\fudos 23}^{n_2}  [x_{\fudos\fudos 23}, x_{\fudos 2}]_c^{n_3}
[[x_{\fudos\fudos 23},x_2]_c,x_{\fudos 2}]_c^{n_4} [[[x_{\fudos\fudos 23},x_2]_c,x_{\fudos 2}]_c,x_3]_c^{n_5} \\
[x_{\fudos\fudos 23}, x_2]_c^{n_6} x_3^{n_7} x_{3\fudos 2}^{n_8} [x_{32},x_{\fudos 2}]^{n_9} x_{32}^{n_{10}}
x_{\fudos 2}^{n_{11}} x_2^{n_{12}} : m_1,m_2\in\I_{0, p-1}, n_i\in\I_{0,1} \}.
\end{align*}
is a basis of $\toba(\lstr(A_2,2))$ and $\dim \toba(\lstr(A_2,2)) = p^22^{12}$. \qed
\end{prop}

\subsubsection{The Nichols algebra $\toba(\lstr(A_{\theta - 1}))$} \label{subsubsection:lstr-a-n}
For details of the following result we refer to \cite[\S 5.3.8]{aah-triang}. In particular, 
one defines $x_{ij}$, $2\le i\le j$, as usual, $\ya_{1\ell} = [x_{\fudos 2}, x_{3\ell}]_c$
and $\ya_{k\ell} = [\ya_{1\ell}, x_{2k}]$, $k>1$. 

\begin{prop} \label{pr:lstr-a-n}
The algebra $\toba(\lstr(A_{\theta - 1}))$ is presented by generators $x_i$, $i\in \Iw_{\theta}$, and relations
\eqref{eq:rels B(V(1,2))-bis}, \eqref{eq:rels B(V(1,2))-x1p}, \eqref{eq:rels B(V(1,2))-x2p-bis},
\eqref{eq:lstr-rels&11disc-1-bis}, \eqref{eq:lstr-rels&11disc-qserre-g=1}, \eqref{eq:lstr-rels&1-1disc-g=1},
\eqref{eq:lstr-rels-ghost-int-trivial}, and
\begin{align}\label{eq:lstr-rels-a-n-diagpart}
&x_{ij}^2 =0, \qquad 2\le i\le j\le\theta,\\
\label{eq:lstr-rels-a-n-diagpart-2}
&[x_{k-1 \, k \, k+1}, x_k]=0, & & 3\le k <\theta.
\\\label{eq:lstr-rels-a-n}
& x_{1j}^2, &  &\ya_{k\ell}^2, & j,k,\ell\in\I_\theta, & k<\ell,
\end{align}

Furthermore there is a PBW-basis in terms of the positive roots of the root system of type $D_{\theta}$
and $\dim \toba(\lstr(A_{\theta - 1})) = p^22^{\theta(\theta - 1)}$. \qed
\end{prop}

\subsection{Realizations}\label{subsec:realizations-block-pts}

Let $H$ be a Hopf algebra, $(g_1, \chi_1, \eta)$ a YD-triple and $(g_j, \chi_j)$, $j\in \I_{2, \theta}$,
a family of YD-pairs for $H$, see \S \ref{subsec:realizations-block}. 
Let $(V, c)$ be a braided vector space  with braiding 
\eqref{eq:braiding-block-several-point}.
Then  
\begin{align}
\cV := \cV_{g_1}(\chi_1, \eta) \oplus \Big(\oplus_{j\in \I_{2, \theta}} \ku_{g_j}^{\chi_j} \Big) \in \ydh
\end{align}
is a \emph{principal realization} of $(V, c)$  over  $H$ if
\begin{align*}
q_{ij}&= \chi_j(g_i),& &i, j\in \I_{\theta};& a_j&= q_{j1}^{-1}\eta(g_j), j\in \I_{2, \theta}.
\end{align*}
Thus $(V, c) \simeq \cV$ 
as braided vector space.
Consequently, if $H$ is finite-dimensional and $(V, c)$ is as in Table \ref{tab:toba-finitedim-block-points},
then $\toba (\cV) \# H$ is a finite-dimensional Hopf algebra.
Examples of finite-dimensional pointed Hopf algebras $A =  \NA(V) \# \ku \Gamma$ with $\Gamma$ abelian
like this are listed in Table \ref{tab:hopf-finitedim-block-points} where the interaction is weak, $\epsilon = 1$
and  $\omega \in \G'_3$.
As in \S \ref{subsec:realizations-block-pt}, $\Gamma$ is a product of $\theta$ cyclic groups, $g_i$ is the $i$-th canonical generator, $\chi_j(g_i)=1$ if $i\neq j$; we set $q_{ij}=1$, $i<j$.

\begin{table}[ht]
\caption{Pointed Hopf algebras from a block and several points} \label{tab:hopf-finitedim-block-points}
\begin{center}
\begin{tabular}{|c|c|c|c|}
\hline   $V_J$ &   $\ghost_J$  &  $\Gamma$ &  $dim A$  \\
\hline
{\scriptsize $\xymatrix{\overset{-1}{\circ} \ar  @{-}[r]^{-1}  
& \overset{-1}{\circ}} \dots \xymatrix{\overset{-1}{\circ} \ar  @{-}[r]^{-1}  & \overset{-1}{\circ} }$}
 & $(1, 0,\dots, 0)$    & {\scriptsize  $\Z/p\times \Z/2p \times \Z/2$}   &  $2^{8}p^4$
\\ 
\cline{4-4}
 &   & {\scriptsize $\Z/p\times \Z/2p \times \big(\Z/2\big)^{\theta-2}$} & $2^{(\theta-1)^2} p^4$
\\ \hline
$\xymatrix{\overset{-1}{\circ} \ar  @{-}[r]^{-1}  & \overset{-1}{\circ}}$  & $(2, 0)$     & {\scriptsize  $\Z/p\times \Z/2p \times \Z/2$} & $2^{14}p^4$
\\ \hline
$\xymatrix{\overset{-1}{\circ} \ar  @{-}[r]^{\omega}  & \overset{-1}{\circ}}$ 
& $(1, 0)$     & {\scriptsize  $\Z/p\times \Z/2p \times \Z/6$} & $2^93^5p^4$
\\ \hline
$\xymatrix{\overset{-1}{\circ} \ar  @{-}[r]^{\omega^2 }  & \overset{\omega}{\circ}}$  
& $(1, 0)$  & {\scriptsize  $\Z/p\times \Z/2p \times \Z/3$} & $2^53^3p^4$
\\ \cline{2-4}
  &  $(0, 1)$   & {\scriptsize  $\Z/p\times \Z/6 \times \Z/3p$} & $2^83^6p^4$
\\ \hline
$\xymatrix{\overset{-1}{\circ} \ar  @{-}[r]^{r^{-1}}  & \overset{r}{\circ}}$, $r \in \G'_N, N > 3$  
&    $(1, 0)$    & {\scriptsize  $\Z/p\times \Z/2p \times \Z/N$} & $2^5N^3p^4$
\\ \hline
$\xymatrix{\overset{-1}{\circ} \ar  @{-}[r]^{\omega}  & \overset{\omega^2}{\circ} \ar  @{-}[r]^{\omega}  & \overset{\omega^2}{\circ} }$ 
& $(1,0, 0)$ & {\scriptsize  $\Z/p\times \Z/2p \times \Z/3 \times \Z/3$} & $2^93^{11}p^4$
\\ \hline
$\xymatrix{\overset{-1}{\circ} \ar  @{-}[r]^{\omega}  & \overset{\omega^2}{\circ} \ar  @{-}[r]^{\omega^2}  & \overset{\omega}{\circ} }$ 
& $(1,0, 0)$ & {\scriptsize  $\Z/p\times \Z/2p \times \Z/3 \times \Z/3$} & $2^93^{11}p^4$
\\ \hline
\end{tabular}
\end{center}

\end{table}

\section{Several blocks, one point}\label{section:YD>3-severalblocks-1pt-poseidon}

Let $t\geq 2$ and $\theta = t +1$. We shall use  the following notation:
\begin{align*}
\Idd_k &= \{k, k + \tfrac{1}{2}\},& k&\in \I_t;
& \Idd &= \Idd_1 \cup \dots \cup \Idd_{t} \cup \{\theta\};
\end{align*}
The Poseidon braided vector space $\pos(\bq,\ghost)$ depends on a datum
\begin{itemize}
\item $\bq\in\ku^{\theta \times \theta}$ such that $\epsilon_i:=q_{ii}=\pm 1$, $q_{ij}q_{ji}=1$ for all $i\neq j \in \I_{\theta}$.
\item $\ghost=(\ghost_j)\in\N^{t}$, $(a_j)$ such that $\ghost_j = \begin{cases} -\sa_j, &\epsilon_j = 1, \\
\sa_j, &\epsilon_j = -1;\end{cases}$ cf. \eqref{eq:discrete-ghost};
\end{itemize}
it has a basis $(x_i)_{i\in\Idd}$ and braiding
\begin{align}\label{eq:braiding-several-blocks-1pt}
c(x_i\ot x_j) &= \left\{ \begin{array}{ll}
q_{ij} \, x_j\otimes x_i, & \lfloor i\rfloor \leq t, \, \lfloor i\rfloor \neq \lfloor j\rfloor, \\
\epsilon_j \, x_j\otimes x_i, & \lfloor i\rfloor =j \leq t, \\
(\epsilon_j \, x_j+x_{\lfloor j\rfloor})\otimes x_i, & \lfloor i\rfloor \leq t, \, j=\lfloor i\rfloor +\frac{1}{2}, \\
q_{\theta j} \, x_j\otimes x_{\theta}, & i = \theta, \, j\in\I_{\theta}, \\
q_{\theta j} \, (x_j+a_j x_{\lfloor j\rfloor})\otimes x_{\theta}, & i = \theta, \, j\notin\I_{\theta}.
\end{array} \right.
\end{align}
We shall use the elements 
\begin{align}\label{eq:several-blocks-yjn}
y_j^{\langle n\rangle} &:=\left\{ \begin{array}{ll} x_j x_{j+\frac{1}{2} \, j}^m, & n=2m+1 \text{ odd};
\\ x_{j+\frac{1}{2} \, j}^m, & n=2m \text{ even}. \end{array} \right. & & j\in\I_t,  n\in\N_0;
\\
\label{eq:defn-sch-several-blocks}
\sch_{\bn} &:= (\ad_c x_{\fudos})^{n_1} \dots (\ad_c x_{t+\frac{1}{2}})^{n_{t}} x_{\theta}, & & \bn = (n_1,\dots,n_{t})\in\N^t_0.
\end{align}

Let $\mathcal A := \{\bn \in\N^t_0: 0 \le \bn \le \ba= (|\sa_1|, \dots, |\sa_t|)\}$,
ordered lexicographically.
For $\bm, \bn \in \cA$ we set
\begin{align*}
\bp_{\bm,\bn} &:= \epsilon_{\theta}\prod_{i,j\in\I_t} q_{ij}^{m_in_j} q_{i\theta}^{m_i} q_{\theta j}^{n_j}
&
\epsilon_{\bn} &:=\bp_{\bn,\bn}
\end{align*}
We also need the following notation:
\begin{align*}
t_+ &:=\{ i\in\I_t: \epsilon_i=1 \}, & t_- &= t-t_+, \\
M_+ &:=\{ \bm\in\cA: \epsilon_{\bm}=1 \}, & M_- &= |\cA|-M_+. 
\end{align*}

The main result of this Section is the following.

\begin{prop} \label{pr:poseidon}
The algebra $\toba(\pos(\bq,\ghost))$ is presented by generators $x_i$, $i\in \Idd$, and relations
\begin{align}\label{eq:poseidon-defrels-Jordan}
&\begin{aligned}
&x_{i+\frac{1}{2} }^p =0, \quad x_i^p =0, 
\\
&x_{i+\frac{1}{2} }x_i -x_ix_{i+\frac{1}{2} }+\frac{1}{2}x_i^2 =0,
\end{aligned} & &i\in\I_t, \, \epsilon_{i}=1;
\\ \label{eq:poseidon-defrels-super-Jordan}
&\begin{aligned}
&x_i^2 = 0, \quad x_{i+\frac{1}{2} \, i}^{2p} =0, \quad x_{i+\frac{1}{2}}^p = 0,\\ 
&x_{i+\frac{1}{2} }x_{i+\frac{1}{2} \, i}- x_{i+\frac{1}{2} \, i}x_{i+\frac{1}{2} } - x_ix_{i+\frac{1}{2} \, i}=0,
\end{aligned} & i&\in\I_t, \, \epsilon_{i}=-1;
\\\label{eq:poseidon-defrels-blocks-commute}
&x_ix_j = q_{ij} \, x_jx_i, & \lfloor i\rfloor &\neq \lfloor j\rfloor  \in\I_t;
\\ \label{eq:poseidon-defrels-q-commute}
&x_ix_{\theta} = q_{i \theta} \, x_{\theta} x_i, & i& \in\I_t;
\\ \label{eq:poseidon-defrels-q-Serre}
&(\ad_c x_{i+\frac{1}{2}})^{1+|\sa_i|}(x_{\theta})=0, & i& \in\I_t,
\\
\label{eq:poseidon-rels-K-1}
&\sch_{\bm}\sch_{\bn} = \bp_{\bm,\bn} \, \sch_{\bn}\sch_{\bm}  &
\bm & \neq \bn \in \cA;
\\ \label{eq:poseidon-rels-K-2}
&\sch_{\bn}^2 =0, & \bn \in \cA, \, \epsilon_{\bn} &= -1,
\\ \label{eq:poseidon-rels-K-3}
&\sch_{\bn}^p =0, & \bn \in \cA, \, \epsilon_{\bn} &= 1.
\end{align}
A basis of $\toba(\pos(\bq,\ghost))$ is given by
\begin{align*}
B =\big\{ y_1^{\langle m_1\rangle} x_{\fudos}^{m_2} \dots y_t^{\langle m_{2t-1} \rangle} x_{t+\frac{1}{2}}^{m_{2t}} \prod_{\bn \in \cA} \sch_{\bn}^{b_{\bn}}: \, &0\le b_{\bn}< 2 \mbox{ if }\epsilon_{\bn}=-1, \\
& 0\le b_{\bn}, m_i < p \mbox{ if }\epsilon_{\bn}=1, \, i\in \I_t \big\}.
\end{align*}
Hence $\dim \toba(\pos(\bq,\ghost)) = 2^{2t_{-} + M_{-}} p^{2t + M_{+}}$.
\qed
\end{prop}

\subsection{Realizations}\label{subsec:realizations-blocks-pt}

Let $H$ be a Hopf algebra, $(g_i, \chi_i, \eta_i)$, $i\in \I_{t}$,
a family of YD-triples and $(g_{\theta}, \chi_{\theta})$ a YD-pair for $H$, see \S \ref{subsec:realizations-block}. 
Let $(V, c)$ be a braided vector space  with braiding 
\eqref{eq:braiding-several-blocks-1pt}.
Then  
\begin{align}
\cV := \Big(\oplus_{i\in \I_{t}} \cV_{g_i}(\chi_i, \eta_i)\Big) \oplus  \ku_{g_\theta}^{\chi_\theta}  \in \ydh
\end{align}
is a \emph{principal realization} of $(V, c)$  over  $H$ if
\begin{align*}
q_{ij}&= \chi_j(g_i),& &i, j\in \I_{\theta};& a_j&= q_{j1}^{-1}\eta(g_j), j\in \I_{t}.
\end{align*}
Thus $(V, c) \simeq \cV$ 
as braided vector space. Consequently, if $H$ is finite-dimensional,
then $\toba (\cV) \# H$ is a finite-dimensional Hopf algebra.

If $q_{ij}=1$ for $i\neq j$, then we may choose $H=\Bbbk \Gamma$, where 
\begin{align*}
\Gamma&=\Z/n_1 \times \dots \Z/n_\theta, &  
n_i & =\begin{cases} p & \text{if }q_{ii}=1, \\ 2p & \text{if }q_{ii}=-1.
\end{cases}
\end{align*}
Thus $\toba(\pos(\bq,\ghost))\# \Bbbk \Gamma$ is a Hopf algebra of dimension $2^{3t_{-} + M_{-}+\delta_{\epsilon_{\theta},-1}} p^{3t + M_{+}+1}$.

\section{A pale block and a point}\label{sec:paleblock}

Let 
$V$ be a braided vector space of dimension 3
with braiding given in the  basis $(x_i)_{i\in\I_3}$ by
\begin{align}\label{eq:braiding-paleblock-point}
(c(x_i \otimes x_j))_{i,j\in \I_3} &=
\begin{pmatrix}
\epsilon x_1 \otimes x_1&  \epsilon x_2  \otimes x_1& q_{12} x_3  \otimes x_1
\\
\epsilon x_1 \otimes x_2 & \epsilon x_2  \otimes x_2& q_{12} x_3  \otimes x_2
\\
q_{21} x_1 \otimes x_3 &  q_{21}(x_2 +  x_1) \otimes x_3& q_{22} x_3  \otimes x_3
\end{pmatrix}.
\end{align}
Let $V_1 = \langle x_1, x_2\rangle$, $V_2 = \langle x_3\rangle$. Let $\Gamma = \Z^2$ with a basis $g_1, g_2$.
We realize $V$ in $\ydG$ by  $V_1 = V_{g_1}$, $V_2 = V_{g_2}$, $g_1\cdot x_1 = \epsilon x_1$, $g_2\cdot x_1 = q_{21} x_1$,
$g_1\cdot x_2 = \epsilon x_2$, $g_2\cdot x_2 = q_{21} (x_2 +  x_1)$, $g_i\cdot x_3 = q_{i2} x_3$.

\smallbreak
As usual, let $\widetilde{q}_{12} = q_{12}q_{21}$; in particular the Dynkin diagram of the braided subspace $\langle x_1, x_3\rangle$ is  $\xymatrix{\overset{\epsilon}  {\circ} \ar  @{-}[r]^{\widetilde{q}_{12}}  & \overset{q_{22}}  {\circ} }$.

\medbreak

As for other cases, we consider
$K=\NA (V)^{\mathrm{co}\,\NA (V_1)}$;
then $K = \oplus_{n\ge 0} K^n$ inherits the grading of $\NA (V)$;
$\NA(V) \simeq K \# \NA (V_1)$ and $K$ is the Nichols algebra of
$K^1= \ad_c\NA (V_1) (V_2)$.
Now $K^1\in {}^{\NA (V_1)\# \Bbbk \Gamma}_{\NA (V_1)\# \Bbbk \Gamma}\mathcal{YD}$ with the adjoint action
and the coaction given by \eqref{eq:coaction-K^1}, i.e.
$\delta =(\pi _{\NA (V_1)\#  \Bbbk \Gamma}\otimes \id)\Delta _{\NA (V)\#  \Bbbk \Gamma}$.
Next we introduce $\sh_{m, n} = (\ad_{c} x_{1})^m(\ad_{c} x_{2})^n x_{3}$; we distinguish two cases:
\begin{align*}
w_m &= (\ad_{c} x_{1})^m x_{3} = \sh_{m,0},&
z_n &= (\ad_{c} x_{2})^n x_3 = \sh_{0, n}.
\end{align*}
By direct computation,
\begin{align}\label{eq:paleblock-1}
g_1\cdot &\sh_{m, n}  =  q_{12} \epsilon^{m+n} \sh_{m, n}, & g_2\cdot w_m  &=  q_{21}^{m}q_{22} w_m,
\\ \label{eq:paleblock-2}
z_{n+1} &= x_2z_n - q_{12}\epsilon^{n} z_n x_2,&
\sh_{m + 1, n} &= x_1\sh_{m, n} - q_{12}\epsilon^{m + n}\sh_{m, n} x_1,
\\ \label{eq:paleblock-3}
\partial_1(&\sh_{m, n}) =0, & \partial_2(\sh_{m, n}) &=0,
\\ \label{eq:paleblock-3.5}
&&\partial_3(w_{m}) &= \prod_{0\le j \le m-1}(1 - \epsilon^j \widetilde{q}_{12}) x_1^m.
\end{align}

\subsection{The block has $\epsilon = 1$}\label{subsubsec:paleblock-case1}
Here $\NA (V_1) \simeq S(V_1)$ is a polynomial algebra, so that $x_1$ and $x_2$ commute, and
\begin{align}\label{eq:paleblock-4}
(\ad_{c} x_{2})^s \sh_{m, n} &= \sh_{m, n + s} & &\text{for all } m,n,s\in\N_0.
\end{align}
Thus $\sh_{m, n}$, $m, n \in \N_0$ generate $K^1$. As in \cite[\S 8.1]{aah-triang}, we have that
\begin{align}\label{eq:paleblock-5}
g_2\cdot \sh_{m, n} &= q_{21}^{m+n}q_{22} \sum_{0\le j \le n} \binom{n}{j}  \sh_{m+j, n-j},
\end{align}

\smallbreak

For $q\in \Bbbk^{\times}$, let $\eny_{p}(q) = V$ be the braided vector space as in \eqref{eq:braiding-paleblock-point} under the assumptions that $\epsilon = 1$,
$q_{12} = q = q_{21}^{-1}$, $q_{22} = - 1$. We call $\NA(\eny_{p}(q))$ and the Nichols algebras $\NA(\eny_{\pm}(q))$, $\NA(\eny_{\star}(q))$ studied in Propositions \ref{prop:paleblock2}, \ref{prop:paleblock3} and \ref{prop:paleblock-case2a} the \emph{Endymion algebras}.

\begin{prop}\label{prop:paleblock1}
The algebra $\cB(\eny_p(q))$ is presented by generators $x_1,x_2, x_3$ and relations
\begin{align}\label{eq:endymion-p-1}
x_1^p&=0, \quad x_2^p=0, \quad x_1x_2 =x_2x_1,
\\ \label{eq:endymion-p-2}
x_1x_3 &= q x_3x_1,
\\ \label{eq:endymion-p-3} 
z_t^2 &=0, \qquad t\in\I_{0,p-1}.
\end{align}
The dimension of $\toba(\eny_p(q))$ is $2^{p}p^{2}$, since it has a PBW-basis
\begin{align*}
B=\{ x_1^{m_1} x_2^{m_2} z_{p-1}^{n_{p-1}} \dots z_{0}^{n_{0}}: n_i \in \{0,1\}, \, m_j \in\I_{0,p-1}\}.
\end{align*}
\end{prop}

\pf 
We claim that $\dim K^1=p$ and $K \simeq \Lambda (K^1)$. In fact, by \eqref{eq:paleblock-3}, \eqref{eq:paleblock-3.5}
and the hypothesis $\widetilde{q}_{12} = 1$,
$w_m =0$ for all $m >0$; thus $\sh_{m, n} =0$ for all $m >0$ by \eqref{eq:paleblock-4}. Using this fact and \eqref{eq:paleblock-5},
\begin{align}\label{eq:paleblock-5bis}
g_2\cdot z_{n} &= q_{21}^{n}q_{22}   z_{n}, & n&\in \N_0.
\end{align}
We have that  for all $n\in\N_0$:
\begin{align}
\label{eq:paleblock-6}
\partial_3(z_n) &= (-1)^n x_1^n,
\\\label{eq:paleblock-7}
\delta(z_n) &= \sum_{0\le j \le n} (-1)^{n+j} \binom{n}{j} x_1^{n-j} g_1^jg_2 \otimes z_j.
\end{align}
Indeed the proof follows as in \cite[Lemma 8.1.4]{aah-triang}. As $x_1^p=0$, we have $z_p=0$. Thus the set $(z_n)_{n \in \I_{0,p-1}}$ is a basis of $K^1$. The braiding is
\begin{align*}
c(z_n \otimes z_s) &=  \sum_{0\le j \le n} (-1)^{n+j} \binom{n}{j} \ad_c (x_1^{n-j} g_1^jg_2) z_s \otimes z_j
= q_{12}^{n}q_{21}^s q_{22}  z_s \otimes z_{n}.
\end{align*}
That is, $K^1$ is of diagonal type and the Dynkin diagram consists of disconnected $p$ points labeled with $-1$. Hence $\NA(K^1) = \Lambda(K^1)$. Moreover, $B$ is a basis of $\toba(V)$ since $\NA(V) \simeq K \# \NA (V_1)$.

The presentation follows as in \cite[Proposition 4.3.7]{aah-triang}.
\epf

\subsection{The block has $\epsilon = -1$}\label{subsubsec:paleblock-case-1}
Here $\NA (V_1) \simeq \Lambda(V_1)$ is an exterior algebra and consequently $\sh_{m, n}$, $m, n \in \{0,1\}$ generates $K^1$.
By direct computation,
\begin{align}
\label{eq:paleblock-8}
g_2 \cdot z_1 &=  q_{21}q_{22}(z_1 + w_1), \qquad \partial_3(z_1) = (1 - \widetilde{q}_{12}) x_2 - \widetilde{q}_{12} x_1,
\\ \label{eq:paleblock-9}
\delta(z_1) &= g_1g_2 \otimes z_1 + \big((1 - \widetilde{q}_{12}) x_2 - \widetilde{q}_{12} x_1\big)g_2 \otimes x_3.
\end{align}

\subsubsection{Case 1: $\widetilde{q}_{12} = 1$} Here
$w_1 =0$  by \eqref{eq:paleblock-3} and \eqref{eq:paleblock-3.5}, so $$ \sh_{1, 1} = -(\ad_{c} x_{2}) w_1  =0.$$
Thus $z_0 = x_3$ and $z_1$ form a basis of $K^1$ and the braiding of $K^1$ is given by
\begin{align}
\label{eq:paleblock-10}
\begin{aligned}
c(x_3 \ot x_3) &= q_{22}x_3 \ot x_3, &  c(x_3 \ot z_1) &= q_{21}q_{22} z_1 \otimes x_3,\\
c(z_1\ot x_3) &= q_{12}q_{22} x_3 \otimes z_1, &  c(z_1 \ot z_1) &=  -q_{22} z_1 \ot z_1.
\end{aligned}
\end{align}
That is, $K^1$ is of diagonal type with Dynkin diagram $\xymatrix{ \overset{q_{22}}{\circ} \ar  @{-}[r]^{q_{22}^2} & \overset{-q_{22}}{\circ} }$.

\smallbreak
For $q\in \Bbbk^{\times}$, let $\eny_{\pm}(q) = V$ be the braided vector space as in \eqref{eq:braiding-paleblock-point} under the assumptions that $\epsilon = -1$,
$q_{12} = q = q_{21}^{-1}$, $q_{22} = \pm 1$.

\begin{prop}\label{prop:paleblock2}
The algebra $\cB(\eny_+(q))$ is presented by generators $x_1,x_2, x_3$ and relations
\begin{align}\label{eq:endymion-1}
x_1^2&=0, \quad x_2^2=0, \quad x_1x_2 =- x_2x_1,
\\ \label{eq:endymion-2}
(x_2x_3 - q x_3x_2)^2&=0, \quad x_3^p=0, 
\\ \label{eq:endymion-3}  x_3(x_2x_3 - q x_3x_2) &=q^{-1} (x_2x_3 - q x_3x_2)x_3,
\\ \label{eq:endymion-4}
x_1x_3 &= q x_3x_1.
\end{align}
Let $z_1 = x_2x_3 - q x_3x_2$. Then $\cB(\eny_+(q))$ has a PBW-basis
\begin{align*}
B=\{ x_1^{m_1} x_2^{m_2} x_3^{m_3}z_{1}^{m_{4}}: m_1, m_2, m_4 \in \{0,1\}, \, m_3 \in\I_{0,p-1}\};
\end{align*}
hence $\dim \cB(\eny_+(q)) = 2^3p$.
\end{prop}

\pf 
Notice that $x_3^p=0$ since $x_3$ is a point labeled with $q_{22}=1$ in $K^1$. Also,
$B$ is a basis thanks to the isomorphism $\toba(\eny_+(q)) \simeq \toba(K^1) \# \toba(V_1)$.
The rest of the proof follows as in \cite[Proposition 8.1.6]{aah-triang}.
\epf

\begin{prop}\label{prop:paleblock3}
The algebra $\cB(\eny_-(q))$ is presented by generators $x_1,x_2, x_3$ and relations \eqref{eq:endymion-1}, \eqref{eq:endymion-4},
\begin{align}
\label{eq:endymion-2b}
x_3 ^2&=0, \quad (x_2x_3 - q x_3x_2)^p=0,
\\ \label{eq:endymion-3b}  x_3(x_2x_3 - q x_3x_2) &= -q^{-1} (x_2x_3 - q x_3x_2)x_3.
\end{align}
Let $z_1 = x_2x_3 - q x_3x_2$. Then $\cB(\eny_-(q))$ has a PBW-basis
\begin{align*}
B=\{ x_1^{m_1} x_2^{m_2} x_3^{m_3}z_{1}^{m_{4}}: m_1, m_2, m_3 \in \{0,1\}, \, m_4 \in\I_{0,p-1}\};
\end{align*}
hence $\dim \cB(\eny_-(q)) = 2^3p$.
\end{prop}

\pf
Notice that $z_1^p=0$ since $z_1$ is a point labeled with $1$ in $K^1$. Also,
$B$ is a basis thanks to the isomorphism $\toba(\eny_-(q)) \simeq \toba(K^1) \# \toba(V_1)$.
The rest of the proof follows as in \cite[Proposition 8.1.6]{aah-triang}.
\epf

\subsubsection{Case 2: $\widetilde{q}_{12}=-1$}
We consider now a fixed choice of $q_{22}$ and $\widetilde{q}_{12}$, which is the corresponding one to the example of finite $\GK$ over a field of characteristic 0.

For $q\in \Bbbk^{\times}$, let $\eny_{\star}(q) = V$ be the braided vector space as in \eqref{eq:braiding-paleblock-point} under the assumptions that
$\epsilon = -1$, $q_{22}=-1$, $q_{12}=q$, $q_{21}=-q^{-1}$.  Recall \eqref{eq:xijk}.

\begin{prop}\label{prop:paleblock-case2a}
The algebra $\cB(\eny_{\star}(q))$ is presented by generators $x_1,x_2, x_3$ and relations \eqref{eq:endymion-1},
\begin{align}\label{eq:endymion-6}
x_3^2=0, \quad x_{31}^2&=0,
\\ \label{eq:endymion-7}
x_2 [x_{23},x_{13}]_c- q^2 [x_{23},x_{13}]_cx_2 &=q \, x_{13}x_{213},
\\ \label{eq:endymion-8}
x_{23}^{2p}=0, \quad [x_{23},x_{13}]_c^p=0, \quad x_{213}^2 &= 0.
\end{align}
Moreover $\cB(\eny_{\star}(q))$ has a PBW-basis
\begin{multline*}
B=\{ x_2^{m_1}x_{23}^{m_2} x_{213}^{m_3}[x_{23},x_{13}]_c^{m_4} x_1^{m_5}x_{13}^{m_6} x_3^{m_7} : \\
m_1, m_3, m_5, m_6, m_7 \in \{0,1\}, \, m_2 \in\I_{0, 2p-1}, m_4 \in\I_{0,p-1}\};
\end{multline*}
hence $\dim \cB(\eny_{\star}(q)) = 2^6p^2$.
\end{prop}

\pf Relations \eqref{eq:endymion-1} are 0 in $\cB(\eny_{\star}(q))$ because $\NA (V_1) \simeq \Lambda(V_1)$; \eqref{eq:endymion-6}
are 0 since $x_1$, $x_3$ generate a Nichols algebra of Cartan type $A_2$ at $-1$.

Notice that $[x_{23},x_{13}]_c=x_{23}x_{13}+x_{13}x_{23}$. By \eqref{eq:endymion-1} and \eqref{eq:endymion-6},
\begin{align}\label{eq:endymion-aux-1}
x_2x_{23}&=-q \, x_{23}x_2, & x_2x_{213}&=q \, x_{213}x_2,
\\ \label{eq:endymion-aux-2}
x_{23}x_3&=-q \, x_3x_{23}, & x_{23}x_1&=-q^{-1}x_1x_{23}-q^{-1}x_{213},
\\ \label{eq:endymion-aux-3}
x_{213}x_1&=q^{-1}x_1x_{213}, & [x_{23},x_{13}]_cx_1&=-q^{-2} x_1[x_{23},x_{13}]_c,
\\ \label{eq:endymion-aux-4}
x_1x_{13}&=-q \, x_{13}x_1, & x_{213}x_3&=q[x_{23},x_{13}]_c-q^2x_3x_{213},
\\ \label{eq:endymion-aux-5}
x_{13}x_3&=-q \, x_3x_{13}, & x_{13}[x_{23},x_{13}]_c &=[x_{23},x_{13}]_c x_{13},
\\ \label{eq:endymion-aux-6}
x_{213}x_{13}&=q \, x_{13}x_{213}, & [x_{23},x_{13}]_c x_3 &=q^2 x_3 [x_{23},x_{13}]_c.
\end{align}
As in \cite[Proposition 8.1.8]{aah-triang} we check that
\begin{align*}
\partial_3(x_{213})&=4x_2x_1 \neq 0, & \partial_3 ([x_{23},x_{13}]_c)&=2q^{-1}x_1x_{13} \neq 0.
\end{align*}

Now we prove that \eqref{eq:endymion-8} holds in $\toba(V)$.
We check that $\partial_i$ annihilates these terms
for $i=1,2,3$.
To simplify the notation, let $u=[x_{23},x_{13}]_c$.
As $\partial_1$, $\partial_2$ annihilate $x_{23}$, $x_{213}$ and $u$, it remains the case $i=3$. 
Using \eqref{eq:endymion-aux-1}-\eqref{eq:endymion-aux-6},
\begin{align*}
\partial_3(x_{213}^2) 
&= 4 \big(q^{-2} x_2x_1 x_{213}+x_{213}x_2x_1 \big)=0,
\\
\partial_3(u^p) &
= 2q^{-1} \sum_{k=0}^{p-1} (-q)^{2+2k-p} u^k x_1x_{13}u^{p-1-k}
= 2q^{-1} p \, u^{p-1}x_1 x_{13}=0.
\end{align*}
For the remaining relation, we 
check that $\partial_3 (x_{23}^{2})=
q^{-1} x_{213}-x_{13} (2x_2+x_1)$ and the following equalities hold (using \eqref{eq:endymion-aux-1}-\eqref{eq:endymion-aux-6}):
\begin{align*}
x_{213}x_{23}^2 &= q(x_{23}^2+u)x_{213},
&
x_{213} u &= q^2 \, u x_{213},
\\
x_{13}x_{23}^2 &=ux_{13}+x_{23}^2x_{13},
&
x_{13} u &= u x_{13},
\\
x_{1}x_{23}^2 &=q^2x_{23}^2x_1 -qx_{13}x_{213},
&
x_{1} u &= q^2 \, u x_{1},
\\
x_{2}x_{23}^2 &= q^2 x_{23}^2x_{2},
&
x_{2} u &= q^2\,  u x_{2}+qx_{13}x_{213},
\end{align*}
Using the previous computations and $x_{13}^2=0$,
\begin{align*}
\big(q^{-1} & x_{213}-2x_{13}x_2-x_{13}x_1\big)(x_{23}^2+\,u) =
q(x_{23}^2+u)x_{213}
-2q^2 x_{13} x_{23}^2x_{2}
\\ & 
\quad -q^2 x_{13}x_{23}^2x_1
+q\, u x_{213}
-2q^2 x_{13} u x_{2}
-q^2 x_{13}u x_{1}
\\ 
&=
qx_{23}^2x_{213} +q \, u x_{213}
-2q^2 (ux_{13}+x_{23}^2x_{13}) x_{2}
\\ & 
\quad -q^2 (ux_{13}+x_{23}^2x_{13}) x_1
+q\, u x_{213}
-2q^2 \, u x_{13}x_{2}
-q^2 \, u x_{13}x_{1}
\\ & 
= q^2(x_{23}^2+2\,u) \big(q^{-1}x_{213}-2x_{13}x_2-x_{13}x_1\big).
\end{align*}
We apply this equality to compute:
\begin{align*}
\partial_3 & (x_{23}^{2p})  =
\sum_{k=0}^{p-1} q^{2k+2-2p} x_{23}^{2k} \partial_3(x_{23}^2) (x_{23}+x_{13})^{2p-2-2k}
\\
&= \sum_{k=0}^{p-1} q^{2k+2-2p} x_{23}^{2k} \big( q^{-1} x_{213}-x_{13} (2x_2+x_1) \big) (x_{23}^2+u)^{p-1-k}
\\
&= \sum_{k=0}^{p-1} x_{23}^{2k}(x_{23}^2+2u)^{p-1-k} \big( q^{-1} x_{213}-x_{13} (2x_2+x_1) \big)
\end{align*}
Using \eqref{eq:endymion-7}, \eqref{eq:endymion-aux-4} and \eqref{eq:endymion-aux-5} we get $u x_{23}^2 = (u+x_{23}^2) u$. Hence $a=x_{23}^2+u$ and $b=u$ satisfy the last equation of \eqref{eq:a-b-conditions}, so \eqref{eq:a-b-formulae} applies and we have
\begin{align*}
\partial_3 & (x_{23}^{2p}) = p(x_{23}^2+u)^{p-1} \big( q^{-1} x_{213}-x_{13} (2x_2+x_1) \big) =0.
\end{align*}

The rest of the proof follows as in \cite[Proposition 8.1.8]{aah-triang}.
\epf

\subsection{Realizations}\label{subsec:realizations-paleblock-pt}

Here we present a realization of a braided vector space as in \eqref{eq:braiding-paleblock-point} over a group algebra $H=\Bbbk\Gamma$, with $\Gamma$ a finite abelian group. We consider $V_1 = \langle x_1, x_2\rangle$, $V_2 = \langle x_3\rangle$. We realize $V$ in $\ydG$ by  $V_1 = V_{g_1}$, $V_2 = V_{g_2}$, $g_1\cdot x_1 = \epsilon x_1$, $g_2\cdot x_1 = q_{21} x_1$,
$g_1\cdot x_2 = \epsilon x_2$, $g_2\cdot x_2 = q_{21} (x_2 +  x_1)$, $g_i\cdot x_3 = q_{i2} x_3$. In all the cases $\Gamma$ will be a product of two cyclic groups, $g_1=(1,0)$, $g_2=(0,1)$. 
Examples of finite-dimensional pointed Hopf algebras 
$A=\toba \big(V_{g_1}\oplus  V_{g_2}\big) \# H$ 
are listed in Table \ref{tab:hopf-finitedim-paleblock-point}.

\renewcommand{\arraystretch}{1.4}

\begin{table}[ht]
\caption{{\small Pointed Hopf algebras $K$ from a pale block and a point}}\label{tab:hopf-finitedim-paleblock-point}
\begin{center}
\begin{tabular}{|c|c|c|c|c|}
\hline $V$ & $(\epsilon, \widetilde{q}_{12}, q_{22})$ & $\Gamma$  & $q_{12}$
& {\small $\dim A$}  \\
\hline
$\eny_{p}(q)$ & $(1,1,-1)$  
& {\scriptsize $\Z/p \times \Z/2p$ } & $1$ & $2^{p+1}p^{4}$
\\\hline
$\eny_{+}(q)$ & $(-1,1,1)$  
& {\scriptsize $\Z/2p \times \Z/p$ } & $1$ & $2^4p^3$
\\\hline
$\eny_{-}(q)$ & $(-1,1,-1)$  
& {\scriptsize $\Z/2p \times \Z/2p$ } & $1$ & $2^5p^3$
\\\hline
$\eny_{\star}(q)$ & $(-1,-1,-1)$  
& {\scriptsize $\Z/2p \times \Z/2p$ } & $\pm 1$ & $2^8 p^4$
\\\hline
\end{tabular}
\end{center}
\end{table}

\end{document}